\documentclass[12pt,english]{amsart}

\input xy
\xyoption{all}

\usepackage{comment}
\usepackage[latin1]{inputenc}
\usepackage{amsfonts,amsmath,amssymb,amsthm}
\usepackage{stmaryrd} 
\usepackage{footmisc}
\usepackage{tikz}
\usetikzlibrary{matrix}
\usepackage[right=3cm, left=3cm]{geometry}
\usepackage{hyperref}

\newcommand{\ra}{\rightarrow}

\newcommand{\Zz}{\mathbb{Z}}
\newcommand{\Qq}{\mathbb{Q}}

\newcommand{\qq}{{\mathbb{Q}}}

\newcommand{\fp}{{\mathfrak{p}}}
\newcommand{\fq}{{\mathfrak{q}}}
\newcommand{\fP}{{\mathfrak{P}}}
\newcommand{\fQ}{{\mathfrak{Q}}}
\newcommand{\mS}{\mathcal{S}} 
{\theoremstyle{plain}
\newtheorem{theorem}{Theorem}[section]    % theorem with number

\newtheorem{twisting lemma}[theorem]{Twisting lemma}
\newtheorem{lemma}[theorem]{Lemma}       % lemma with number
\newtheorem{proposition}[theorem]{Proposition}  % lemma with number
\newtheorem{corollary}[theorem]{Corollary}   % lemma with number

\newtheorem{question}[theorem]{Question}
}
{\theoremstyle{remark}
\newtheorem{definition}[theorem]{Definition}      % lemma with number
\newtheorem{remark}[theorem]{Remark}   % theorem without number

}
\newcommand\oline[1] {{\overline{#1}}}

\DeclareMathOperator\Gal{Gal}

\def\Frac{\hbox{\rm Frac}}

\def\cm{\hbox{\hbox{\rm C}\kern-5pt{\raise 1pt\hbox{$|$}}}}

\def\lhfl#1#2{\smash{\mathop{\hbox to 12mm{\leftarrowfill}}
\limits^{#1}_{#2}}}
\def\rhfl#1#2{\smash{\mathop{\hbox to 12mm{\rightarrowfill}}
\limits^{#1}_{#2}}}

\def\build#1_#2^#3{\mathrel{
\mathop{\kern 0pt#1}\limits_{#2}^{#3}}}

\def\htrait#1#2{\smash{\mathop{\hbox to 12mm{\hrulefill}}
\limits^{#1}_{#2}}}

\def\sxbullet{{\raise 2pt\hbox{\bf .}}}
\numberwithin{equation}{section}

\begin{document}

\title[The local behaviour of specializations]{On the local behaviour of specializations of function field extensions}

\author{Joachim K{\"{o}}nig}

\email{koenig.joach@technion.ac.il}

\address{Department of Mathematics, Technion - Israel Institute of Technology, Haifa 32000, Israel}

\author{Fran\c cois Legrand}

\email{legrandfranc@technion.ac.il}

\address{Department of Mathematics, Technion - Israel Institute of Technology, Haifa 32000, Israel}

\author{Danny Neftin}

\email{dneftin@technion.ac.il}

\address{Department of Mathematics, Technion - Israel Institute of Technology, Haifa 32000, Israel}

\date{\today}

\keywords{Galois theory, local behaviour, specializations, Grunwald problems, crossed products, parametric sets} 

\maketitle

\begin{abstract}
Given a field $k$ of characteristic zero and an indeterminate $T$ over $k$, we investigate the local behaviour at primes of $k$ of finite Galois extensions of $k$ arising as specializations of finite Galois extensions $E/k(T)$ (with $E/k$ regular) at points $t_0 \in \mathbb{P}^1(k)$. We provide a general result about decomposition groups at primes of $k$ in specializations, extending a fundamental result of Beckmann concerning inertia groups. We then apply our result to study crossed products, the Hilbert--Grunwald property, and finite parametric sets.
\end{abstract}

\section{Introduction} \label{sec:intro}

\subsection{Grunwald problems} \label{sec:intro_grunwald}

Given a number field $k$, {\textit{Grunwald problems}} concern the structure of completions $L_\fp/k_\fp$ of {\textit{$G$-extensions}} $L/k$ at any finitely many given primes $\fp$ of $k$. Here, a $G$-extension $L/k$ is a finite Galois extension of $k$ with Galois group $G$, and $k_\fp$ (resp., $L_\fp$) denotes the completion of $k$ (resp., of $L$) at $\fp$ (resp., at a prime $\fP$ of $L$ lying over $\fp$)\footnote{Recall that the latter is independent of the choice of $\fP$ (up to $k_\fp$-isomorphism).}. The original motivation for such problems arose from their key role in the structure theory of finite dimensional division algebras over number fields. % see \cite[Chapters 14-19]{Pie82}. 
Since then, they have received much attention, especially due to connections with the regular inverse Galois problem; see \cite{DG11,DG12}, and with weak approximation; see, e.g., \cite{Har07, Dem10, LA14, DLN17}. 

More precisely, given a finite set $\mathcal{S}$ of primes of $k$,  and given finite Galois extensions $L^{(\fp)}/k_\fp$, $\fp \in \mathcal{S}$,  with Galois groups embedding into $G$, the Grunwald problem $(G, (L^{(\fp)}/k_\fp)_{\fp\in \mathcal{S}})$ asks whether there is a $G$-extension $L/k$ whose completion $L_\fp$ at $\fp$ is $k_\fp$-isomorphic to $L^{(\fp)}$, $\fp \in  \mathcal{S}$. If such an extension $L/k$ exists, it is called a {\textit{solution}} to the Grunwald problem $(G, (L^{(\fp)}/k_\fp)_{\fp\in \mathcal{S}})$. We note that, instead of prescribing the local extensions $L_{\fp}/k_\fp$, $\fp\in \mathcal{S}$, weaker versions ask for the existence of a $G$-extension $L/k$ with prescribed local degrees  $[L_\fp:k_\fp]$, $\fp\in \mathcal{S}$, or with prescribed local Galois groups $\Gal(L_\fp/k_\fp)$, $\fp\in \mathcal{S}$. These weaker versions usually suffice for applications to classical problems, see, e.g., \cite[Corollary 2]{Wan50} and \cite{Sch68}. 

Examples of Grunwald problems $(G,(L^{(\fp)}/k_\fp)_{\fp\in \mathcal{S}})$ with no solution $L/k$ occur already for cyclic groups $G$, when $\mathcal{S}$ contains a prime of $k$ lying over $2$ \cite{Wan48}. However, it is expected \cite[\S 1]{Har07} that, for solvable groups $G$, every Grunwald problem $(G,(L^{(\fp)}/k_\fp)_{\fp\in \mathcal{S}})$ has a solution, provided $\mathcal{S}$ is disjoint from some finite set $\mS_{\rm{exc}}$ of ``exceptional" primes of $k$, depending only on $G$ and $k$. This is known when (1) $G$ is abelian, and $\mS_{\rm{exc}}$ is the set of primes of $k$ dividing $2$ \cite[(9.2.8)]{NSW08}; (2) $G$ is an iterated semidirect product $A_1 \rtimes (A_2 \rtimes \cdots \rtimes A_n)$ of finite abelian groups, and   $\mathcal{S}_{\rm{exc}}$ is the set of primes of $k$ dividing $|G|$; see \cite[Th\'eor\`eme 1]{Har07} and \cite[Theorem 1.1]{DLN17}; (3) $G$ is solvable of order prime to the number of roots of unity in $k$, and $\mS_{\rm{exc}}=\emptyset$ \cite[(9.5.5)]{NSW08}; and (4) there exists a {\textit{generic extension}} for $G$ over $k$, and $\mS_{\rm{exc}}=\emptyset$ \cite[Theorem 5.9]{Sal82}. Among the above, the latter is the only method which applies to non-solvable groups. However, the family of non-solvable groups for which a generic extension is known is quite restricted, e.g., it is unknown whether the alternating group $A_n$ has a generic extension for $n\ge 6$. See \cite{JLY02} for an overview on generic extensions. 

The main source of realizations of non-solvable groups $G$ over $k$ is via {\textit{$k$-regular $G$-extensions}}, that is, via $G$-extensions $E/k(T)$, where $T$ is an indeterminate over $k$ and $k$ is algebraically closed in $E$. Indeed, by {\textit{Hilbert's irreducibility theorem}}, every non-trivial $k$-regular $G$-extension $E/k(T)$ has infinitely many linearly disjoint {\textit{specializations}} $E_{t_0}/k$, $t_0\in \mathbb{P}^1(k)$, with Galois group $G$. Many groups have been realized by this method; see, e.g., \cite{MM99}, and references within, as well as \cite{Zyw14} for more recent examples.

This specialization process provides a natural way to attack Grunwald problems for {\textit{$k$-regular Galois groups}}, that is, for finite groups $G$ admitting a $k$-regular $G$-extension of $k(T)$. Namely, given such an extension $E/k(T)$, it is natural to ask for the local behaviour of specializations $E_{t_0}/k$, $t_0\in \mathbb{P}^1(k)$. That is,  which local extensions $L^{(\fp)}/k_\fp$, which local Galois groups $\Gal(L^{(\fp)}/k_\fp)$, and which local degrees $[L^{(\fp)}:k_\fp]$ arise by completing the specialization $E_{t_0}/k$ at primes $\fp$ of $k$, when $t_0$ runs over $ \mathbb{P}^1(k)$? For points $t_0\in \mathbb{P}^1(k)$ which are $\fp$-adically far from {\textit{branch points}} of $E/k(T)$, this approach was deeply investigated  by D\`ebes and Ghazi \cite{DG11, DG12}\footnote{See \cite[\S1.6]{DG12} for a review of related previous results.}, and applies only to unramified local extensions. Namely, given a $k$-regular $G$-extension $E/k(T)$, a finite set $\mathcal{S}$ of primes of $k$, disjoint from some finite set $\mS_{\rm{exc}}:=\mS_{\rm{exc}}(E/k(T))$  (depending only on $E/k(T)$), and unramified extensions $L^{(\fp)}/k_\fp$ with Galois group embedding into $G$, $\fp\in \mathcal{S}$, \cite{DG12} provides $t_0\in k$ such that $E_{t_0}/k$ is a solution  to  the Grunwald problem $(G,(L^{(\fp)}/k_\fp)_{\fp\in \mathcal{S}})$. 

\subsection{Main results} \label{sec:intro_local_behaviour}

The goal of this paper is to study the local behaviour at a prime $\fp$ of $k$ of specializations $E_{t_0}/k$, when $t_0$ is $\fp$-adically close to a branch point of $E/k(T)$.

\subsubsection{Background}% on inertia groups of specializations}

A first related conclusion can be derived from the algebraic cover theory of Grothendieck. Namely, if $\fp$ is not in some finite set $\mathcal{S}_{\rm{exc}}:=\mS_{\rm{exc}}(E/k(T))$ of primes of $k$ and if $\fp$ ramifies in the specialization $E_{t_0}/k$, then, 
$t_0$ and some branch point $t_i$ of $E/k(T)$ {\textit{meet modulo $\fp$}} \footnote{Note that the set $\mathcal{S}_{\rm{exc}}$ is chosen such that every $t_0$ meets at most one branch point modulo $\fp$.%, see \S\ref{sec:main_statement}.
}. In the special case where $t_i$ is $k$-rational and $v_\fp(t_0)$ is non-negative, this means that $a_\fp:=v_\fp(t_0-t_i)$ is positive, where $v_\fp$ denotes the normalized $\fp$-adic valuation. 
A fundamental theorem by Beckmann \cite{Bec91, Con00} then asserts that the inertia group $I_{t_0,\fp}$ of $E_{t_0}/k$ at $\fp$ is determined by $a_\fp$ and the inertia group $I_{t_i}$ at a fixed prime of $E$ lying over the prime $\mathcal{P}_i$ of $k[T-t_i]$ generated by $T-t_i$. 
Namely, 
%Moreover, letting $I_{t_i}$ denote the geometric inertia group of a fixed prime of $E$ lying over the prime $\mathcal{P}_i$ of $k[T-t_i]$ generated by $T-t_i$,  a classical theorem of Beckmann \cite{Bec91, Con00} asserts that 
$I_{t_0,\fp}$ %the inertia group of $E_{t_0}/k$ at $\fp$ 
is conjugate to $I_{t_i}^{a_\fp}$. We refer to \S\ref{sec:basics_inertia} for more details.

\subsubsection{Decomposition groups of specializations} \label{sec:intro_thm_main}

Given $t_0\in \mathbb{P}^1(k)$, assumed to meet the branch point $t_i$ modulo $\fp$, we show that the entire  decomposition group at $\fp$ of the specialization $E_{t_0}/k$ is determined by the local behaviour at $t_i$, thus extending Beckmann's theorem. 
Namely, suppose for simplicity that $t_i$ is $k$-rational, and let $D_{t_i}$ denote the geometric decomposition group at a fixed prime of $E$ lying over the geometric prime $\mathcal{P}_i$. Recall that the specialization $E_{t_i}/k$ has Galois group $D_{t_i}/I_{t_i}$; let $\varphi_i:D_{t_i}\ra D_{t_i}/I_{t_i}$ be the natural projection. Let $D_{t_i,\fp}$ be the decomposition group at a prime $\fP$ of $E_{t_i}$ lying over $\fp$; this is a subgroup of $D_{t_i}/I_{t_i}$. Note that, up to conjugation, the subgroup $\varphi_i^{-1}(D_{t_i,\fp})$ of $G$ is independent of the choice of an arithmetic prime $\fP$ lying over $\fp$ and a geometric prime lying over $\mathcal{P}_i$. 

\begin{theorem} \label{thm:main}
There exists a finite set $\mathcal{S}_{\rm{exc}}:=\mathcal{S}_{\rm{exc}}(E/k(T))$ of primes of $k$ that satisfies the following property.
Suppose that the given prime $\fp$ of $k$ is outside $\mathcal{S}_{\rm{exc}}$. Moreover, suppose that the given branch point $t_i$ is $k$-rational, that $t_0$ and $t_i$ meet modulo $\fp$, and that the exponent $a_\fp$ is coprime to $|I_{t_i}|$. Then, the decomposition group of $E_{t_0}/k$ at $\fp$ is conjugate by an element of $G$ to $\varphi_i^{-1}(D_{t_i,\fp})$.
\end{theorem}
\noindent
Note that %$\mathcal{S}_{\rm{exc}}$ contains the primes $\fp$ dividing the order of $G$, and hence 
$\varphi_i^{-1}(D_{t_i,\fp})$ is the group generated by $I_{t_i}$ and a lift of the Frobenius of $E_{t_i}/k$ at $\fp$.  We refer to Theorem \ref{thm main} for a more general version of Theorem \ref{thm:main}, where we relax the assumptions on $t_i$ and $a_\fp$, and which is stated over more general base fields.

\subsubsection{Solving Grunwald problems} \label{sec:intro_fields}

Theorem \ref{thm:main} shows that the decomposition group of $E_{t_0}/k$ at $\fp$ is determined by the local data $(\varphi_i,D_{t_i,\fp})$ at $t_i$, when $t_0$ and $t_i$ meet modulo $\fp$. Theorem \ref{thm:fields} below shows that this is the only constraint on completions of $E_{t_0}/k$ at $\fp$ for such specialization points $t_0$. 

\begin{theorem} \label{thm:fields}
%There exists a finite set $\mathcal{S}'_{\rm{exc}}$ of primes of $k$, depending only on $E/k(T)$, satisfying the following. 
Suppose that $\mathcal{S}$ is a finite set of primes of $k$ disjoint from some finite set of primes $\mathcal{S}'_{\rm{exc}}:=\mS'_{\rm{exc}}(E/k(T))$. For each $\fp\in\mathcal{S}$, fix a $k$-rational branch point $t_{i(\fp)}$ of $E/k(T)$ and a finite Galois extension $L^{(\fp)}/k_\fp$ with Galois group (resp., inertia group) $\varphi_{i(\fp)}^{-1}(D_{t_{i(\fp)},\fp})$ (resp., $I_{t_{i(\fp)}}$). Then, there exists $t_0\in k$ such that $E_{t_0}/k$ is a solution to the Grunwald problem $(G,(L^{(\fp)}/k_\fp)_{\fp\in \mathcal{S}})$. 
\end{theorem}

\noindent
See Theorem \ref{thm fields} for a more general version of Theorem \ref{thm:fields}, where the assumption on the branch points is relaxed, and where $k$ is not necessarily a number field.

\begin{comment}
Note that in Theorem \ref{thm:fields}, $t_0$ necessarily meets $t_{i(\fp)}$ modulo $\fp$ with an exponent coprime to the order of $I_{t_{i(\fp)}}$.
\end{comment}

Given a single homomorphism $\varphi_i:D_{t_i}\ra D_{t_i}/I_{t_i}$ at a ($k$-rational) branch point $t_i$ of $E/k(T)$, Theorem \ref{thm:fields} provides a solution to all Grunwald problems $(G, (L^{(\fp)}/k_\fp)_{\fp\in \mathcal{S}})$, 
where, for each $\fp$ in $\mathcal{S}$, the local extension $L^{(\fp)}/k_\fp$ has inertia group $I_{t_i}$ and decomposition group $\varphi_i^{-1}(D_{t_i,\fp})$. 
Such results are especially applicable to problems where the primes are allowed to vary and the decomposition groups are fixed; see, in particular, \S\ref{sec:intro_admiss}. 

\subsubsection{About the proof}
Let $R$ be the ring of integers of $k$. Our approach to Theorems \ref{thm:main} and \ref{thm:fields} considers the fraction field $F$  of the two-dimensional domain  $R_\fp[[T-t_i]]$, which is also the completion of $R[T]$ at the ideal generated by $\fp$ and $T-t_i$. We use a theorem of Eisenstein, recalled as Theorem \ref{thm:eisen}, to show that, for all but finitely many primes $\fp$ of $k$,  the group $\Gal(E\cdot F/F)$ is determined by the local data $(\varphi_i,D_{t_i,\fp})$ at $t_i$. On the other hand, we show that, by reducing the extension $E\cdot F/F$ at a place extending the place $(T-t_0)$ of $k(T)$, the resulting extension of residue fields is the completion of $E_{t_0}/k$ at $\fp$, giving the desired connection between $D_{t_0,\fp}$ and the local data at $t_i$. See \S\ref{sec:main_proof}. 
The above method is considerably different from that of D\`ebes and Ghazi, which, to our knowledge, does not apply to ramified extensions; see \cite[\S1.6]{DG12}, and which is based on specializations of extensions of $k_\fp(T)$.

We use the theory of Brauer embedding problems to determine the extensions $M/k_\fp$ which satisfy the constraints imposed by Theorem \ref{thm:main} on completions of specializations $(E_{t_0})_\fp$, and the above reduction process to show that such fields $M$ are obtained as specializations $(E_{t_0})_\fp$. 
%Given an extension of $M/k_\fp$ whose Galois group is $\phi^{-1}(D_{t_i,\fp})$ (as imposed by Theorem \ref{thm:fields}), we use the theory of Brauer embedding problems to . 

%We then compare the reductions at $(T-t_i)$ and $(T-t_0)$. 
 %Our method relies on a classical result of Eisenstein (recalled as Theorem \ref{thm:eisen} in the sequel) for choosing primes $\fp$ of $k$ at which the specialization process at $t_0$ does not degenerate. To prove Theorem \ref{thm:fields}, we also make use of Brauer embedding problems.

\subsection{Applications} \label{sec:applications} Theorem \ref{thm:main} is then used to study crossed product division algebras, the Hilbert--Grunwald property, and finite parametric sets. 

\subsubsection{Prescribing decomposition groups and crossed products} \label{sec:intro_admiss}
Recall that the Galois group $D$ of a tamely ramified extension of $k_\fp$ is {\textit{metacyclic}}, that is,  $D$ admits a cyclic normal subgroup $I$ such that $D/I$ is cyclic. 
The following application of Theorem \ref{thm:main} ensures the appearance of certain metacyclic subgroups of $G$ as decomposition groups of specializations of the $k$-regular $G$-extension $E/k(T)$.

\begin{theorem} \label{cor:reoc}
Given a branch point $t_i$ of $E/k(T)$, let $D_{t_i}$ (resp., $I_{t_i}$) denote the decomposition (resp., inertia) group at $t_i$, and let $\tau$ be an element of $D_{t_i}$. Then, there exist infinitely many primes $\fp$ of $k$, and, for each such prime $\fp$, infinitely many $t_0\in k$ such that $E_{t_0}/k$ is a $G$-extension with decomposition group $\langle I_{t_i},\tau\rangle$ at $\fp$. 
\end{theorem} 

\noindent
We refer to Theorem \ref{chebotarev} for a generalization of Theorem \ref{cor:reoc}, where we consider finitely many primes of $k$ at the same time.
 
We then apply this result to prove the existence of {\textit{$G$-crossed product}} division algebras over number fields for various finite groups $G$. Recall that a finite dimensional division algebra over its center $k$ is a {\textit{$G$-crossed product}} if it admits a maximal subfield $L$ that is finite and Galois over $k$, and such that ${\rm{Gal}}(L/k) = G$. A $G$-crossed product division algebra is equipped with an explicit structure, which plays a key role in the theory of central simple algebras. By Schacher's result \cite[Theorem 4.1]{Sch68}, the existence of a $G$-crossed product division algebra with center $\mathbb Q$ implies that $G$ has metacyclic Sylow subgroups. The converse is a long standing open conjecture, originating in \cite{Sch68}: for every finite group $G$ with metacyclic Sylow subgroups, there exists a $G$-crossed product division algebra with center $\mathbb Q$. Although this conjecture has been extensively studied, cf.~\cite[Section 11.A]{ABGV11}, the only finite non-abelian simple groups for which the conjecture is known to hold are $A_5 \cong {\rm{PSL}}_2(\mathbb{F}_4) \cong {\rm{PSL}}_2(\mathbb{F}_5)$ \cite[Theorem 1]{GS79}, $A_6 \cong {\rm{PSL}}_2(\mathbb{F}_9)$ \cite[Theorem 6]{FS91}, $A_7$ \cite[Theorem 6]{FS91}, ${\rm{PSL}}_2(\mathbb{F}_7)$ \cite[Proposition 5]{AS01}, ${\rm{PSL}}_2(\mathbb{F}_{11})$ \cite[Theorem 2]{AS01}, and the Mathieu group M$_{11}$ \cite{KN17}. 

Given a $k$-regular $G$-extension $E/k(T)$ with suitable properties, we use Theorem \ref{cor:reoc} to establish a general criterion for the existence of specialization points $t_0\in k$ such that $E_{t_0}$ is a maximal subfield of a $G$-crossed product. We refer to Theorem \ref{admiss_crit} for our precise result. We then use this criterion to derive the first infinite family of finite non-abelian simple groups $G$ with a $G$-crossed product division algebra over $\Qq$.

\begin{theorem} \label{thm:admis}
Let $p$ be a prime number such that either $p \equiv 3 \pmod 8$ or $p\equiv 5 \pmod 8$. Then, there exists a ${\rm{PSL_2}}(\mathbb{F}_p)$-crossed product division algebra with center $\Qq$.
\end{theorem}

\noindent
See Theorem \ref{thm:psl1} where we show more generally that Theorem \ref{thm:admis} holds over an arbitrary number field $k$ (instead of $\Qq$), provided primitive 4-th roots of unity are not in $k$. See also Theorem \ref{thm:psl2} where a similar conclusion is shown to hold for prime numbers lying in other arithmetic progressions.

\subsubsection{On the Hilbert-Grunwald property} \label{sec:intro_results_Hilbert-Grunwald}

Following a terminology due to D\`ebes and Ghazi, and motivated by the partial positive answer provided by Theorem \ref{thm:fields}, one may ask whether a given $k$-regular $G$-extension $E/k(T)$ has the {\textit{Hilbert-Grunwald property}}, that is, whether there exists a finite set $\mathcal{S}_{\rm{exc}}$ of primes of $k$ such that every Grunwald problem $(G, (L^{(\fp)}/k_\fp)_{\fp\in \mathcal{S}})$, with $\mathcal{S}$ disjoint from $\mathcal{S}_{\rm{exc}}$, has a solution inside the set of specializations of $E/k(T)$. Recall that this property is always fulfilled, provided we restrict to unramified Grunwald problems \cite{DG11, DG12}.

In contrast, our results on the local behaviour of (ramified) specializations show in particular that, for many finite groups $G$, no $k$-regular $G$-extension of $k(T)$ has the Hilbert-Grunwald property: %More precisely, we obtain the following:

\begin{theorem} \label{neg_grunwald_boohoo}
Assume that $G$ has a non-cyclic abelian subgroup. Then, no $k$-regular $G$-extension of $k(T)$ has the Hilbert--Grunwald property. 
%Then there exists no finite set $\mathcal{S}_{\rm{exc}}$ of primes of $k$ such that the union of the sets of all specializations of any given finitely many $k$-regular $G$-extensions of $k(T)$ contains the set of solutions to all Grunwald problems $(G,(L^{(\fp)})/k_\fp)_{\fp\in \mathcal{S}})$, with $\mathcal{S}$ disjoint from $\mathcal{S}_{\rm{exc}}$.  
\end{theorem}

\noindent
We refer to Theorem \ref{thm:neggrunwald} for a more general version with finitely many $k$-regular $G$-extensions at the same time. 
%Obvious examples of finite groups $G$ with a non-cyclic abelian subgroup are non-cyclic finite abelian groups, dihedral groups $D_n$ (of order $2n$) with $n$ even, symmetric groups $S_n$ with $n \geq 4$, alternating groups $A_n$ with $n \geq 4$. More generally, 
Groups with a non-cyclic abelian subgroup have been classified in the literature relatively explicitly; see, e.g., the classical papers \cite{Zas35} and \cite{Suz55}. Examples of such groups contain, in addition to the obvious  $S_n$ for $n\geq 4$ and $D_n$ with $n$ even,  all non-abelian simple groups; see the proof of Corollary \ref{simple}. In particular, these groups satisfy the conclusion of Theorem \ref{neg_grunwald_boohoo}. Recall that there are many non-abelian simple groups $G$ such that the only known realizations of $G$ over $k$ are specializations of $k$-regular $G$-extensions of $k(T)$, and only finitely many $k$-regular $G$-extensions of $k(T)$ are known\footnote{Of course, up to obvious manipulations such as changes of variable, or translates by rational extensions $k(s)/k(t)$, both of which do not yield new specializations.}. For such finite groups $G$, Theorem \ref{neg_grunwald_boohoo} implies that one cannot solve all Grunwald problems for the group $G$ over the number field $k$ by using only the currently known realizations of $G$ over $k$.

\subsubsection{Non-existence of finite parametric sets} \label{sec:intro_param}

As in \cite{KL18}, given a finite group $G$, we call a set $S$ of $k$-regular $G$-extensions of $k(T)$ {\textit{parametric}} if every $G$-extension of $k$ occurs as a specialization of some extension $E/k(T)$ in $S$; see Definition \ref{def para}. If $S$ consists of a single extension $E/k(T)$, the extension $E/k(T)$ is called  {\textit{parametric}}. This is a generalization of the {\textit{Beckmann-Black problem}} (over $k$), which, with our phrasing, asks whether every finite group $G$ has a parametric set over $k$ \footnote{See, e.g., the survey paper \cite{Deb01b} for more on the Beckmann-Black problem.}. Here, we ask whether a given finite group $G$ has a {\textit{finite}} parametric set over $k$.

On the one hand, a given finite group $G$ has a parametric extension over $k$, provided $G$ has a {\textit{one parameter generic polynomial}} over $k$. However, this latter condition is very restrictive. For example, in the case $k=\Qq$, only the subgroups of the symmetric group $S_3$ have a one parameter generic polynomial; see \cite[\S2.1 \& \S8.2]{JLY02}\footnote{ Note that no finite group $G$ with a finite parametric set over $k$, but with no one parameter generic polynomial over $k$, is available in the literature.}. 
On the other hand, no finite group $G$ was known to have no finite parametric set over the given number field $k$, until a recent joint work by the first two authors \cite{KL18}. However, the ``global" strategy developed in that paper requires $G$ to have a non-trivial proper normal subgroup which satisfies some further properties. In particular, this cannot be used for finite simple groups. 

As a further application of our ``local" results, we obtain the first examples of finite non-abelian simple groups without finite parametric sets: %over number fields:

\begin{theorem} \label{parametric_boohoo}
Let $n \geq 4$ be an integer. Then, the alternating group $A_n$ has no finite parametric set over $k$.
\end{theorem}

\noindent
We refer to Theorem \ref{grunwald_nonparam} where we show more generally that a given finite group $G$ has no finite parametric set over the number field $k$, provided $G$ has a non-cyclic abelian subgroup and some very weak Grunwald property holds. This last property holds in particular if $G=A_n$, by  Mestre \cite{Mes90}; see Corollary \ref{examples}.

\vspace{3mm}

{\textbf{Acknowledgments.}} This work is partially supported by the Israel Science Foundation (grant No. 577/15), and  the  Technion startup grant. 

\section{Notation and preliminaries} \label{sec:basics} 

The section is organized as follows. \S\ref{sec:basics_ded} is devoted to standard background on De-dekind domains, while we recall classical material on function field extensions in \S\ref{sec:basics_functionfields}.

\subsection{Basics on Dedekind domains} \label{sec:basics_ded}

For more on below, we refer to \cite{Ser79}.

\subsubsection{Residue, localization, and completion} \label{sec:basics_localization}

Let $R$ be a domain of characteristic zero, and $k$ its fraction field. Given a {\textit{prime}} $\fp$ of $k$, i.e., given a non-zero prime ideal $\fp$ of $R$, the {\textit{residue field}} $\overline{k}_\fp$ of $R$ at $\fp$ is the fraction field ${\rm{Frac}}(R/\fp)$ of $R/\fp$. {\textit{We shall assume that $\overline{k}_\fp$ is perfect.}} The local ring $R_{(\fp)}:=\{x/y \, : \, (x,y) \in R^2 \, \, {\rm{and}} \, \, y \notin \fp\}$ is the {\textit{localization}} of $R$ at $\fp$. The unique maximal ideal of $R_{(\fp)}$ is $\fp R_{(\fp)}$, and the corresponding residue field $R_{(\fp)} / \fp R_{(\fp)}$ is canonically isomorphic to $\overline{k}_\fp$. If $\fp R_{(\fp)}$ is principal, there is a discrete valuation $v_\fp$ on $k$ whose valuation ring is $R_{(\fp)}$. Let $k_\fp$ be the {\textit{completion}} of $k$ with respect to $v_\fp$. There is a unique valuation on $k_\fp$ extending $v_\fp$, the residue field of which coincides with $\oline k_\fp$ (up to canonical isomorphism). 

\subsubsection{Dedekind domains} \label{sec:basics_ramification}

From now on, we assume that $R$ is a Dedekind domain. 

Let $L/k$ be a finite extension. The integral closure $S$ of $R$ in $L$ is also a Dedekind domain. Let $\fP$ be a {prime} of $L$ lying over $\fp$. The residue field $\overline{L}_\mathfrak{P}$ ($=S/\fP$) is a finite extension of  $\overline{k}_\fp$. The {\textit{residue degree}} of $L/k$ at $\mathfrak{P}$ is the degree $f_{\fP|\fp}:=[\overline{L}_\fP:\overline{k}_\fp]$ of this finite extension. The maximal positive integer $e:=e_{\fP|\fp}$ for which $\fP^e$ is contained in $\fp S$ is the {\textit{ramification index}} of $\fP$ in $L/k$. The prime $\fp$ is {\textit{ramified}} in $L/k$ if there exists a prime $\fP$ lying over it with ramification index $e_{\fP|\fp}>1$, and {\textit{unramified}} otherwise. It is {\textit{totally ramified}} if $e_{\fP|\fp}=[L:k]$ for the unique prime $\fP$ lying over $\fp$, and {\textit{totally split}} in $L/k$ if $\fp$ is unramified in $L/k$, and if $f_{\fP|\fp}=1$ for every prime $\fP$ of $L$ lying over $\fp$. 

From now on, we assume that the extension $L/k$ is Galois with Galois group $G$. The subgroup of $G$ which consists of all elements $\sigma$ such that $\sigma(\mathfrak P)= \mathfrak P$ is the {\textit{decomposition group}} $D_{\mathfrak P}$ of $L/k$ at $\mathfrak P$. The residue extension $\overline{L}_\mathfrak{P}/\overline{k}_\fp$ is Galois, and the restriction map $D_{\mathfrak P} \rightarrow {\rm{Gal}}(\overline{L}_\mathfrak{P}/\overline{k}_\fp)$ is an epimorphism, whose kernel is the {\textit{inertia group}} $I_{\mathfrak P}$ of $L/k$ at ${\mathfrak P}$. The decomposition groups $D_{\mathfrak P_1}$ and $D_{\mathfrak P_2}$ (resp., the inertia groups $I_{\mathfrak P_1}$ and $I_{\mathfrak P_2}$) of two primes $\fP_1$ and $\fP_2$ lying over $\fp$ in $L/k$ are conjugate in $G$. When a prime $\fP$ lying over $\fp$ is fixed, we set $D_\fp:=D_{\fP}$ and $I_\fp:=I_{\fP}$. One has $I_\fp \trianglelefteq D_\fp$, the cardinalities $|I_\fp|$ and $|D_\fp|$ are independent of the choice of the prime $\fP$ lying over $\fp$, and are equal to $e_{\fP|\fp}$ and $e_{\fP|\fp} f_{\fP|\fp}$, respectively. Similarly, the residue extension $\overline{L}_{\fP}/\overline{k}_\fp$ of $L/k$ at $\fP$ does not depend on the choice of $\fP$ (up to $\overline{k}_\fp$-isomorphism); we therefore denote it by $\overline{L}_{\fp}/\overline{k}_\fp$.  
The following basic lemma describes explicitly the residue field $\overline{L}_\fp$.
\begin{lemma} \label{dedekind}
Let $f(X) \in R[X]$ be a monic separable polynomial of positive degree $n$ and splitting field $L$ over $k$. Denote the roots of $f(X)$ by $y_1,\dots,y_n$. Let $\fp$ be a prime of $k$ such that the discriminant of $f$ is not contained in $\fp$. Then, one has $\overline{L}_\fp= \overline{k}_\fp(\overline{y_1},\dots,\overline{y_n}),$ where $\overline{x}$ denotes the reduction of $x \in L$ modulo a prime of $L$ lying over $\fp$.
\end{lemma}
The completion $L_\fP$ of $L$ at ${\mathfrak P}$ is Galois over $k_\fp$, with Galois group canonically isomorphic to $D_\fP$ (via restriction from $L_\fP$ to $L$). One has $L_\fP = L \cdot k_\fp$, where $L \cdot k_\fp$ denotes the compositum of $L$ and $k_\fp$ inside a given algebraic closure $\widetilde{k_\fp}$ of $k_\fp$. Since, up to $k_\fp$-isomorphism, this does not depend on the choice of $\fP$, one can speak of the {\textit{completion}} of $L/k$ at $\fp$, and denote it by $L_\fp / k_\fp$. The maximal subextension of ${L_\fp} / {k_\fp}$ which is unramified at $\fp$ equals $L_\fp^{I_\fp}/k_\fp$, where $L_\fp^{I_\fp}$ denotes the fixed field of $I_\fp$ in $L_\fp$; we denote it by $L_\fp^{\rm{ur}} / {k_\fp}$, and call it the {\textit{unramified part}} of ${L_\fp} / {k_\fp}$. Furthermore, fixing a $k$-embedding $\sigma$ of $L$ into $\widetilde{k_\fp}$ induces a choice of a prime $\fP$ lying over $\fp$ in $L/k$, via $\fP:=\{x\in L\, :\, v_\fp(\sigma(x)) >0\}$, where $v_\fp$ is the $\fp$-adic valuation on $k_\fp$, extended to $\widetilde{k_\fp}$.

\subsection{Extensions of function fields} \label{sec:basics_functionfields}

Let $\widetilde k$ denote a fixed algebraic closure of a given field $k$ of characteristic zero, and let $T$ be an indeterminate over $k$. 

\subsubsection{Classical background}

Let $E/k(T)$ be a finite Galois extension with Galois group $G$, and such that $E/k$ is {\textit{regular}}, that is, such that $E \cap \widetilde{k}=k$. We then say that $E/k(T)$ is a {\textit{$k$-regular $G$-extension.}} 

A point $t_0$ in $\mathbb{P}^1(\widetilde{k})$ is {\textit{a branch point}} of $E/k(T)$ if the prime ideal $(T-t_0) \, \widetilde{k}[T-t_0]$ of $\widetilde{k}[T-t_0]$ is ramified in  $E \cdot \widetilde{k}/\widetilde{k}(T)$ \footnote{Replace $T-t_0$ by $1/T$ if $t_0=\infty$.
The compositum $E \cdot \widetilde{k}$ of $E$ and $\widetilde{k}(T)$ is taken in a fixed algebraic closure of $k(T)$ containing $\widetilde{k}$.}. The extension $E/k(T)$ has finitely many branch points, denoted by $t_1,\dots,t_r$ (one has $r=0$ if and only if $G$ is trivial). For each $i \in \{1,\dots,r\}$, denote the decomposition group (resp., the inertia group) of $E(t_i)/k(t_i)(T)$ at (a fixed prime of $E(t_i)$ lying over) $(T-t_i) \, k(t_i)[T-t_i]$ by $D_{t_i}$ (resp., by $I_{t_i}$).

Given $t_0 \in \mathbb{P}^1(k)$, the residue extension of $E/k(T)$ at $(T-t_0) \, k[T-t_0]$ is denoted by ${E}_{t_0}/k$, and called the {\textit{specialization of $E/k(T)$ at $t_0$}}. It is a finite Galois extension whose Galois group is canonically isomorphic to $D_{t_0}/I_{t_0}$, where $D_{t_0}$ (resp., $I_{t_0}$) denotes the decomposition group (resp., the inertia group) of $E/k(T)$ at $(T-t_0) \, k[T-t_0]$.

If $k$ is the fraction field of a Dedekind domain $R$, we denote the decomposition group (resp., the inertia group) of the specialization $E_{t_0}/k$ at a given non-zero prime ideal $\fp$ of $R$ such that the residue field $\overline{k}_\fp$ is perfect by $D_{t_0,\fp}$ (resp., by $I_{t_0,\fp}$). Note that $D_{t_0,\fp}$ is canonically isomorphic to a subgroup of $D_{t_0}/I_{t_0}$.

We recall the following well-known lemma, cf.~\cite{PV05}, which is a standard consequence of Krasner's lemma and the compatibility between the Hilbert specialization property and the weak approximation property of $\mathbb{P}^1$:

\begin{lemma}\label{PV}
Assume $k$ is a number field. % with ring of integers $R$. 
Let $\mathcal{S}$ be a finite set of primes of $k$, and $P(T,Y) \in k[T][Y]$ a monic, separable polynomial with splitting field $E$ over $k(T)$. 
For each $\fp \in \mathcal{S}$, choose $t_\fp$ in $k$ such that $P(t_\fp,Y)$ is separable, and let $(E_{t_\fp})_\fp/k_\fp$ be the completion of $E_{t_\fp}/k$ at $\fp$. 
Then, there exist infinitely many $t_0 \in k$ such that $E_{t_0}/k$ has Galois group $G$ and its completion at $\fp$ is $(E_{t_\fp})_\fp/k_\fp$ for each $\fp \in \mathcal{S}$. 
Moreover, one may require these specializations $E_{t_0}/k$ to be pairwise linearly disjoint.
\end{lemma} 

\subsubsection{Laurent series fields and their algebraic extensions} \label{sec:basics_laurent}

Given $t_0 \in \mathbb{P}^1(k)$, set $\fp:=(T-t_0) \, k[T-t_0]$. The unramified part $E_\fp^{\rm{ur}}/k((T-t_0))$ of the completion $E_\fp/k((T-t_0))$ of $E/k(T)$ at $\fp$ is of the form $E_{t_0}((T-t_0)) / k((T-t_0))$; see \cite[page 55]{Ser79}. 
Furthermore, one has $E_\fp = E_\fp^{{\rm{ur}}}(\sqrt[e]{\pi})$ for some element $\pi \in E_\fp^{{\rm{ur}}}$ of $\fp$-adic valuation $1$. Since $\pi = \alpha(T-t_0)(1+\beta)$ for some $\alpha\in E_{t_0}$ and $\beta\in \fp E_\fp^{{\rm{ur}}}$, and since $1+\beta$ is an $e$-th power in $E_\fp^{{\rm{ur}}}$, we may replace $\pi$ by $\alpha(T-t_0)$, so that $E_\fp = E_\fp^{{\rm{ur}}}(\sqrt[e]{\alpha(T-t_0)})$. 
Moreover, since $E_\fp$ is the splitting field of $X^e-\alpha(T-t_0)$, the $e$-th roots of unity are contained in $E_\fp$, and then even in $E_\fp \cap \widetilde{k} = E_{t_0}$, giving the following lemma. 
%In particular, one has Lemma \ref{branchcycle} below, which will be used on several occasions in the sequel:
%Fix $i\in\{1,\ldots,r\}$, set $k'=k(t_i)$ and $E'=E(t_i)$. % at $(T-t_i)$, and $E'^{\rm{ur}}_\fp$ its maximal unramified extension. 
\begin{lemma} \label{branchcycle}
Let $t_i$ be a branch point of $E/k(T)$ of ramification index $e_i$. Then, the specialization $E(t_i)_{t_i}$ %of $E(t_i)/k(t_i)(T)$ at $t_i$ 
contains all $e_i$-th roots of unity.
\end{lemma}
%\begin{proof}
%Set $\fp_i:=(T-t_i) \, k(t_i)[T-t_i]$, and consider the completion $E(t_i)_{\fp_i}/k(t_i)((T-t_i))$ of $E(t_i)/k(t_i)(T)$ at $\fp_i$. As above, the field $E(t_i)_{\fp_i}$ is the splitting field over $(E(t_i)_{\fp_i})^{{\rm{ur}}}$ of the polynomial $X^{e_i}-\alpha(T-t_i)$ for some $\alpha\in (E(t_i))_{t_i}$. Hence, all $e_i$-th roots of unity are contained in $E(t_i)_{\fp_i}\cap \widetilde{k} = (E(t_i))_{t_i}$, as needed for the lemma.
%\end{proof}
Setting $k':=k(t_i)$, $E':=E(t_i)$, and $\fp_i:=(T-t_i) \, k'[T-t_i]$, 
the completion $E'_{\fp_i}$ is then a solution to the embedding problem %given by $1\to I_{t_i}\to 
$D_{t_i}\stackrel{\varphi}{\to} \Gal(E'^{\rm{ur}}_{\fp_i}/k'((T-t_i)))$, where the image is % ${\rm{Gal}}(E(t_i)_{\fp_i}^{\rm{ur}}/k(t_i)((T-t_i)))$ is 
identified with $D_{t_i}/I_{t_i}$ and $\varphi$ is the natural projection. 
Since $E'_{\fp_i}$ is a Kummer extension of $E'_{t_i}((T-t_i))$,  the conjugation action of $D_{t_i}$ on $I_{t_i}$ is isomorphic to the action on the group of $e_i$-th roots of unity in $E'_{t_i}$. Such an embedding problem is then called a {\textit{Brauer embedding problem}}. See, e.g., \cite[Chapter IV, \S7]{MM99} for more details.

Finally, we will make use of a theorem about the structure of algebraic power series, which was first stated over the integers by Eisenstein. For our purposes, let $R$ be an integral domain with fraction field $k$. The general version below can be found, e.g., in \cite[Lemma 2.1]{BBC12}.

\begin{theorem} \label{thm:eisen}
Let $\alpha = \sum_{n=0}^\infty \alpha_n T^n \in k[[T]]$ be algebraic over $k(T)$. Then, there exist $r$ and $s$ in $R$ such that $r\cdot (\alpha_n s^n) \in R$ for each $n\ge 0$. In particular, if $R$ is a Dedekind domain, there exist only finitely many prime ideals $\fp$ of $R$ such that $\alpha$ is not in $R_{(\fp)}[[T]]$.
\end{theorem}

\section{Ramification in specializations} \label{sec:basics_inertia}

In this section, we recall some classical properties on ramification in specializations of function field extensions.

Let $k$ be the fraction field of a Dedekind domain $R$ of characteristic zero, and let $\fp$ be a non-zero prime ideal of $R$ such that $\overline{k}_\fp$ is perfect. Let $v_\fp$ denote the discrete valuation on $k$ with valuation ring $R_{(\fp)}$. Let $T$ be an indeterminate over $k$ and $G$ a finite group.

First, we recall the definition of {\textit{meeting modulo $\fp$}}:

\begin{definition} \label{meeting} 
%\begin{itemize}
%\item[(1)] 
(1) Let $F/k$ be a finite extension, $R_F$ the integral closure of $R$ in $F$, and $\fp_F$ a non-zero prime ideal of $R_F$. We say that two distinct points $t_0$ and $t_1$ in $\mathbb{P}^1(F)$ {\textit{meet modulo $\fp_F$}} if either $v_{\fp_F}(t_0) \geq 0$, $v_{\fp_F}(t_1) \geq 0$, and $v_{\fp_F}(t_0-t_1) > 0$, or $v_{\fp_F}(t_0) \leq 0$, $v_{\fp_F}(t_1) \leq 0$, and $v_{\fp_F}((1/t_0) - (1/t_1)) > 0$ (where $v_{\fp_F}$ denotes the discrete valuation on $F$ associated with $\fp_F$)\footnote{We set $1/\infty = 0$, $1 / 0 = \infty$, $v_{\fp}(\infty) = -\infty$, and $v_{\fp}(0) = \infty$.}.

%\item[(2)] 
\noindent (2) We say that two distinct points $t_0$ and $t_1$ in $\mathbb{P}^1(\widetilde{k})$ {\textit{meet modulo $\fp$}} if $t_0$ and $t_1$ meet modulo a prime ideal of the integral closure of $R$ in $k(t_0,t_1)$ lying over $\fp$.
%\end{itemize}
\end{definition}

Note that, if $t_0$ and $t_1$ meet modulo $\fp$ and $F/k$ is a finite extension containing $k(t_0,t_1)$, then, $t_0$ and $t_1$ meet modulo a prime ideal of the integral closure of $R$ in $F$ lying over $\fp$.

In the case where $t_0$ is $k$-rational and meets $t_1$ modulo $\fp$, the following lemma asserts the existence of a unique degree $1$ prime lying over $\fp$ at which $t_0$ and $t_1$ meet. % at primes of $k(t_1)$ lying over $\fp$.
\begin{lemma} \label{lemma meeting}
For every $t_1\in \mathbb{P}^1(\widetilde{k})$, there exists a finite set $\mathcal S_1$ of primes of $k$, depending only on $t_1$, which satisfies the following property. Suppose $\fp\not\in\mathcal S_1$ and let $t_0\in \mathbb P^1(k)\setminus \{t_1\}$ be such that $t_0$ and $t_1$ meet modulo $\fp$. Then, there exists a unique prime  $\fp':=\fp'(t_0,t_1,\fp)$ lying over $\fp$ in $k(t_1)/k$ with residue degree $f_{\fp'\vert\fp}=1$ at which $t_0$ and $t_1$ meet. 
%\begin{enumerate}
%\item $t_0$ and $t_1$ meet modulo $\fp'$,
%\item the residue degree $f_{\fp'\vert\fp}$ equals $1$.
%\end{enumerate}
\end{lemma}

\begin{remark} \label{remark meeting} \
%\begin{itemize}
%\item[(1)]
(1) If $t_1$ is $k$-rational, then, one has $\fp'(t_0,t_1,\fp)=\fp$.
%\item[(2)] 

\noindent (2) For each prime ideal $\fp'$ lying over $\fp \not \in \mathcal{S}_{1}$ in $k(t_1)/k$ with residue degree $f_{\fp' | \fp}=1$, the weak approximation property of $\mathbb{P}^1$ provides infinitely many $t_0 \in k$ such that $\fp'(t_0,t_1,\fp)=\fp'$.
%\end{itemize}
\end{remark}

\begin{proof}
By part (1) of Remark \ref{remark meeting}, we may assume $t_1 \ne \infty$. We require the set $\mathcal S_1$ to contain the (finite) set of primes $\mathfrak{q}$ of $k$ such that the minimal polynomial $m_{t_1}(T)$ of $t_1$ over $k$ is either not integral at $\mathfrak{q}$ or not separable modulo $\mathfrak{q}$. In particular, for every prime $\fp'$ of $k(t_1)$ lying over $\fp$, we may assume $v_{\fp'}(t_1)\geq 0$ and $\oline{k(t_1)}_{\fp'}$ is generated over $\oline k_\fp$ by the reduction modulo $\fp'$ of $t_1$, by Lemma \ref{dedekind}.

For the existence part, note that, since $t_0$ and $t_1$ meet modulo $\fp$, the above provides a prime ideal $\fp'$ lying over $\fp$ in $k(t_1)/k$ such that $v_{\fp'}(t_0) \geq 0$, $v_{\fp'}(t_1) \geq 0$, and $v_{\fp'}(t_0-t_1) > 0$. Since $\fp$ is not in $\mS_1$, and since the reduction modulo $\fp'$ of $t_1$ equals the reduction $\overline{t_0}$ modulo $\fp'$ of $t_0$, we have $\oline{k(t_1)}_{\fp'} = \oline k_\fp(\oline{t_0}) = \oline k_\fp$, that is, $f_{\fp' | \fp}=1$.

For the uniqueness part, assume that $t_0$ and $t_1$ meet modulo two distinct primes $\fp'_1$ and $\fp'_2$ of $k(t_1)$ lying over $\fp$, both with residue degree $1$. Since $m_{t_1}(T)$ is separable modulo $\fp$, the primes $\fp'_1$ and $\fp'_2$  correspond to distinct linear factors $T-a_1$ and $T-a_2$ of the reduction of $m_{t_1}(T)$ in $\oline k_\fp[T]$, so that $t_1\equiv a_j$ modulo $\fp'_j$ for each $j \in \{1,2\}$. As $t_0$ and $t_1$ meet modulo $\fp'_j$, the difference $t_0-t_1$ is in $\fp'_j$. Hence, $t_0- a_j$ is in $\fp'_j \cap k= \fp$. We then get $a_1\equiv a_2$ modulo $\fp$, which cannot happen.
\end{proof}

Now, we recall the definition of {\textit{intersection multiplicity at $\fp$}}.
Below, the minimal polynomial over $k$ of any given element $t_1 \in \mathbb{P}^1(\widetilde{k})$ is denoted by $m_{t_1}(T)$ (we set $m_{t_1}(T)=1$ if $t_1 = \infty$). Denote the constant coefficient of $m_{t_1}(T)$ by $a_{t_1}$. Then, the minimal polynomial of $1/t_1$ over $k$ is 
\begin{itemize}
\item $m_{1/t_1}(T)=(1/a_{t_1})\, T^{{\rm{deg}}(m_{t_1}(T))} \, m_{t_1}(1/T)$ if $t_1 \in \widetilde{k} \setminus \{0\}$,
\item $m_{1/t_1}(T)=1$ if $t_1=0$,
\item $m_{1/t_1}(T)=T$ if $t_1 = \infty$.
\end{itemize}

Let $t_1$ be in $\mathbb{P}^1(\widetilde{k})$ and $t_0$ in $\mathbb{P}^1(k)$. Assume $v_{\fp}(a_{t_1})=0$ if $t_1 \not=0$ to make the intersection multiplicity well-defined in Definition \ref{**} below.

\begin{definition} \label{**}
{\textit{The intersection multiplicity $I_{\fp}(t_0,t_1)$ of $t_0$ and $t_1$ at $\fp$}} is 
$$I_{\fp}(t_0,t_1)= \left \{ \begin{array} {ccc}
          v_{\fp}(m_{t_1}(t_0)) & {\rm{if}} & v_{\fp}(t_0) \geq 0, \\
          v_{\fp}(m_{1/t_1}(1/t_0)) & {\rm{if}} &  v_{\fp}(t_0) \leq 0. \\   
          \end{array} \right.$$
\end{definition}

Lemma \ref{equiv} below, which is \cite[Lemma 2.5]{Leg16c}, connects  Definitions \ref{meeting} and \ref{**}.

\begin{lemma} \label{equiv} 
Let $t_1$ be in $\mathbb{P}^1(\widetilde{k})$ and $t_0$ in $\mathbb{P}^1(k)$. Assume $v_{\fp}(a_{t_1})=0$ if $t_1 \not=0$.
\begin{itemize}
\item[(1)] If $I_{\fp}(t_0,t_1) >0$, then, $t_0$ and $t_1$ meet modulo $\fp$.
\item[(2)] The converse in {\rm{(1)}} holds, provided $m_{t_1}(T)$ is in $R_{(\fp)}[T]$.
\end{itemize}
\end{lemma}

Finally, we recall the following classical result on ramification in specializations; see, e.g., \cite[Proposition 4.2]{Bec91}, \cite{Con00}, %\cite[Th\'eor\`eme 1.3.3]{Flo02}, 
and \cite[\S2.2.3]{Leg16c}. 

Let $E/k(T)$ be a $k$-regular $G$-extension with branch points $t_1,\dots,t_r$. For every $i \in \{1, \dots ,r\}$, denote the inertia group of $E(t_i)/k(t_i)(T)$ at $(T-t_i)\, k(t_i)[T-t_i]$ by $I_{t_i}$.

\vspace{2.5mm}

\noindent
{\textbf{Specialization Inertia Theorem.}} 
{\textit{Assume that $\fp$ is not in some finite set $\mathcal{S}_{\rm{bad}}:= \mS_{\rm{bad}}(E/k(T))$ of prime ideals of $R$ \footnote{See, e.g., \cite[\S2.2.3]{Leg16c} for an explicit description of these ``bad" primes.}. Let $t_0 \in \mathbb{P}^1(k) \setminus \{t_1,\dots,t_r\}$. 
\begin{itemize}
\item[(1)] If $\fp$ ramifies in $E_{t_0}/k$, then, $t_0$ meets some branch point $t_i$ of $E/k(T)$ modulo $\fp$.
\item[(2)] Suppose that $t_0$ and $t_i$ meet modulo $\fp$. Then, $I_{t_0,\fp}$ is conjugate to $I_{t_i}^{I_{\fp}(t_0,t_i)}$. 
\end{itemize}}}

\section{Main results} \label{sec:main}
In this section, we state and prove Theorems \ref{thm main}, \ref{thm fields}, and \ref{chebotarev}, which are the most general versions of Theorems \ref{thm:main}, \ref{thm:fields}, and \ref{cor:reoc}, respectively. 
%The present section is organized as follows. \S\ref{sec:main_statement} is devoted to the statement of Theorem \ref{thm main}, which is is the main tool provided in this paper to describe decomposition groups of specializations, and which is a generalization of Theorem \ref{thm:main} stated in \S\ref{sec:intro_thm_main}. Theorem \ref{thm main} is then proved in \S\ref{sec:main_proof}. Next, in \S\ref{sec:main_applications}, we state and prove Theorem \ref{thm fields}, of which Theorem \ref{thm:fields} stated in \S\ref{sec:intro_fields} is a special case, and which shows that certain local extensions satisfying the necessary group theoretical conditions provided by Theorem \ref{thm main} occur as completions of suitable specializations. Finally, \S\ref{sec:main_chebotarev} is devoted to Theorem \ref{chebotarev}, our most general result on re-occurence of inertia and decomposition groups in specializations, which is a generalization of Theorem \ref{cor:reoc} stated in \S\ref{sec:intro_admiss}.
Throughout this section, we use the notation of \S\ref{sec:basics}-\ref{sec:basics_inertia}. Let $k$ be the fraction field of a Dedekind domain $R$ of characteristic zero such that, for every prime ideal $\fp$ of $R$, the residue field $\overline{k}_\fp$ is perfect. Given an indeterminate $T$ over $k$ and a finite group $G$, let $E/k(T)$ be a $k$-regular $G$-extension with branch points $t_1,\dots,t_r$. Fix $i\in\{1,\ldots,r\}$. Recall that $D_{t_i}$ (resp., $I_{t_i}$) is the decomposition group (resp., the inertia group) of $E(t_i)/k(t_i)(T)$ at $(T-t_i) \, k[T-t_i]$, and that $D_{t_i,\fp'}$ is the decomposition group of $(E(t_i))_{t_i}/k(t_i)$ at a prime $\fp'$ of $k(t_i)$, so that $D_{t_i,\fp'}$ is canonically identified with a subgroup of $D_{t_i}/I_{t_i}$. Set $e_i:=|I_{t_i}|$, and let $\varphi:D_{t_i} \rightarrow D_{t_i}/I_{t_i}$ be the natural projection.

\subsection{Statement of Theorem \ref{thm main}} \label{sec:main_statement}

\begin{theorem} \label{thm main}
Assume that $t_0\in \mathbb{P}^1(k) \setminus \{t_1,\dots,t_r\}$ and $t_i$ meet modulo a prime $\fp$ of $k$ avoiding a finite set $\mathcal{S}_{\rm{exc}} := \mS_{\rm{exc}}(E/k(T))$ of primes of $k$. Let $\fp'=\fp'(t_0,t_i,\fp)$ be the unique prime ideal lying over $\fp$ in $k(t_i)/k$ provided by Lemma \ref{lemma meeting}\footnote{To make this well-defined, we assume in particular that $\mathcal{S}_{\rm{exc}}$ contains the set $\mathcal{S}_{1}$ from Lemma \ref{lemma meeting}.}.
\begin{itemize}
\item[(1)] The decomposition group $D_{t_0,\fp}$ is conjugate in $G$ to a subgroup $U$ of $D_{t_i}$ such that $\varphi(U) = D_{t_i, \fp'}$. Moreover, if $I_\fp(t_0,t_i)$ is coprime to $e_i$, one has $U=\varphi^{-1}(D_{t_i, \fp'})$. 
\item[(2)] The unramified part of the completion $E_{t_0}\cdot k_\fp$ of $E_{t_0}$ at $\fp$ contains $(E(t_i))_{t_i}\cdot k(t_i)_{\fp'}$, with equality if $I_\fp(t_0,t_i)$ is coprime to $e_i$.
\end{itemize}
\end{theorem}

As the specialization $E_{t_0}/k$ is Galois, the compositum $E_{t_0}\cdot k_\fp$ is independent of the embedding of $E_{t_0}$ into a given algebraic closure of $k_\fp$. Similarly, $(E(t_i))_{t_i}\cdot k(t_i)_{\fp'}$ is independent of the embedding of $(E(t_i))_{t_i}$ into a given algebraic closure of $k(t_i)_{\fp'}$. 

\begin{remark} \label{rk thm main} \
%\begin{itemize}
%\item[(1)] 
(1) Let $\fp$ be a prime of $k$, not in $\mathcal{S}_{\rm{exc}}$, and let $\fp'$ be a prime lying over $\fp$ in $k(t_i)/k$ with residue degree $f_{\fp' | \fp}=1$. As in part (2) of Remark \ref{remark meeting}, there exist infinitely many $t_0\in k$ such that $t_0$ and $t_i$ meet modulo $\fp'$, and such that $I_\fp(t_0,t_i)$ is coprime to $e_i$. Thus, as a consequence of part (1) of Theorem \ref{thm main} and of the Specialization Inertia Theorem, there exist infinitely many $t_0 \in k \setminus \{t_1,\dots,t_r\}$ such that $I_{t_0, \fp}$ is conjugate to $I_{t_i}$, and $D_{t_0,\fp}$ is conjugate to  $\varphi^{-1}(D_{t_i,\fp'})$. % under $\varphi$.
%\item[(2)]

\vspace{0.5mm}

 \noindent 
(2) Part (2) of Theorem \ref{thm main} implies that the residue degree of $(E(t_i))_{t_i}/k(t_i)$ at $\fp'$ divides that of $E_{t_0}/k$ at $\fp$, with equality if $I_\fp(t_0,t_i)$ is coprime to $e_i$. 
%\item[(3)] 

\vspace{0.5mm}

\noindent (3) The assumption that $I_\fp(t_0,t_i)$ is coprime to $e_i$ is necessary in general for the more precise conclusions in parts (1) and (2) of Theorem \ref{thm main}, as the following easy example shows. Set $k:=\Qq$, $E:=\Qq(\sqrt{T})$, $t_i:=0$, and let $p$ be a prime number. Since $E/\Qq(T)$ is totally ramified at $t_i$, one has $E_{t_i}=\Qq$, and $|I_{t_i}|=|D_{t_i}|=2$. Set $t_0:=\alpha\cdot p^2$, where $\alpha$ denotes a rational number of $p$-adic valuation $0$. Then, one has $({E_{t_0}})_p=\Qq_p$ if $\alpha$ is a square modulo $p$ (and hence $1=|D_{t_0,p}|<|\varphi^{-1}(D_{t_i,p})|=2$), whereas $(E_{t_0})_p/\Qq_p$ has degree $2$ if $\alpha$ is not a square modulo $p$, and therefore $E_{t_i}\cdot \Qq_p \subsetneq E_{t_0}\cdot \Qq_p$ in this case. 
%\end{itemize}
\end{remark}

\subsection{Proof of Theorem \ref{thm main}} \label{sec:main_proof}

By possibly changing the variable $T$, we may assume that $t_i$ is not equal to $\infty$, and that $t_i$ is integral over $R$. For simplicity, set $k':=k(t_i)$, $k'_{\fp'} :=k(t_i)_{\fp'}$, $E':=E(t_i)$, $E'_{t_i}:=(E(t_i))_{t_i}$,  $S:=T-t_i$, and let $R'$ be the integral closure of $R$ in $k'$. From now on, we fix an embedding of $E$ into a given algebraic closure $\widetilde{k'_{\fp'}((S))}$ of $k'_{\fp'}((S))$. Every compositum of fields below has to be understood inside $\widetilde{k'_{\fp'}((S))}$. 

We break the proof into four steps.

\subsubsection{Step I} \label{sec:main_completion} \label{Step I}

Here, we identify the Galois group of $E'\cdot k'((S))/k'((S))$ with $D_{t_i}$ and that of $E'\cdot k'_{\fp'}((S))/k'_{\fp'}((S))$ with a subgroup of $D_{t_i}$, namely, with $\varphi^{-1}(D_{t_i,\fp'})$. %completing at $S$, and at $\fp$} 

First, consider the compositum of $E' $ and $k'((S))$. The restriction of the Galois group $\Gal(E'\cdot  k'((S))/k'((S)))$ to $E'/k'(S)$ preserves a prime $\mathfrak{Q}$ of $E'$ lying over the prime generated by $S$. Hence, we identify $\Gal(E'\cdot  k'((S))/k'((S)))$ with the decomposition group $D_{t_i}$ of $E'/k'(S)$ at $\mathfrak{Q}$. Moreover, $E'_{t_i}/k'$ is the residue extension of $E'/k'(S)$ at $\mathfrak{Q}$, and the unramified part of $E' \cdot k'((S))/k'((S))$ is $E'_{t_i} \cdot k'((S))/ k'((S))$; see \S\ref{sec:basics_laurent}. Thus, $E'\cdot  k'((S))/ E'_{t_i} \cdot k'((S))$ is totally ramified at the prime generated by $S$, and its Galois group is identified with the inertia group $I_{t_i}$ of $E'/k'(S)$ at $\mathfrak{Q}$. 

Now, consider the compositum $E'_{t_i}\cdot k'_{\fp'}$. This defines an embedding of $E'_{t_i}$ into the algebraic closure of $k'_{\fp'}$ that is contained in $\widetilde{k'_{\fp'}((S))}$, and hence a prime $\fP'$ of $E'_{t_i}$ lying over $\fp'$ such that the restriction of $\Gal(E'_{t_i}\cdot k'_{\fp'}/k'_{\fp'})$ to $E'_{t_i}/k'$ preserves $\fP'$. Therefore, we identify $\Gal(E'_{t_i}\cdot k'_{\fp'}/k'_{\fp'})$ with the decomposition group $D_{t_i,\fp'}$ of $E'_{t_i}/k'$ at $\fP'$.

Next, consider the compositum $E' \cdot k'_{\fp'}((S))$. Its Galois group $V$ over $k'_{\fp'}((S))$ is identified with a subgroup of $D_{t_i} = \Gal(E'\cdot k'((S))/k'((S)))$ via restriction. Note that, as $E' \cdot k'((S)) / E'_{t_i}\cdot k'((S))$ is totally ramified, the fields $E' \cdot k'((S))$ and $E'_{t_i}\cdot k'_{\fp'}((S))$ are linearly disjoint over $E'_{t_i}\cdot k'((S))$. Hence, $\Gal(E'\cdot k'_{\fp'}((S))/E'_{t_i}\cdot k'_{\fp'}((S)))$ is identified with $I_{t_i} = \Gal(E'\cdot k'((S))/E'_{t_i}\cdot k'((S)))$, so that $\varphi(V) = D_{t_i,\fp'}$. 

We then obtain the following diagram of inclusions and Galois groups:
\begin{equation}\label{equ:diag1}
\xymatrix{
E' \ar@{-}[dd] \ar@{-}[r] & E' \cdot k'((S)) \ar@{-}[d]^{I_{t_i}} \ar@{-}@/_2.7pc/[dd]_{D_{t_i}}  \ar@{-}[r] &  E' \cdot k'_{\fp'}((S)) \ar@{-}[d]_{I_{t_i}} \ar@{-}@/^2.9pc/[dd]^{V}  \\
&  E'_{t_i}\cdot k'((S)) \ar@{-}[d] \ar@{-}[r] &  E'_{t_i}\cdot k'_{\fp'}((S)) \ar@{-}[d]_{D_{t_i,\fp'}} \\
k'(S)  \ar@{-}[r] &  k'((S))  \ar@{-}[r] &  k'_{\fp'}((S))
}
\end{equation}

\subsubsection{Step II} \label{Step II}

Our reduction process modulo $(T-t_0)$ will occur in the domain $R'_{\fp'}[[S]]$ \footnote{This domain is the completion of $R'[S]$ at the maximal ideal generated by $\fp'$ and $S$; see, e.g., \cite[Exercise 7.11]{Eis95}.}; see {Step III} (\S\ref{Step III}). Let $F$ be the fraction field of $R'_{\fp'}[[S]]$; note that $F\subseteq k'_{\fp'}((S))$. 
Here, we show that $E'_{t_i}\cdot F$ is contained in $E'\cdot F$, and identify the Galois groups of $E'_{t_i}\cdot F/F$ and $E'\cdot F/F$ with $D_{t_i,\fp'}$ and $\varphi^{-1}(D_{t_i,\fp'})$, respectively.
%we describe the unramified part and the Galois group of the extension $E'\cdot F/F$. Namely, the latter is identified with $\varphi^{-1}(D_{t_i,\fp'})$. 

This description rests on the following lemma:

% Consider the complete domain $R'_{\fp'}[[S]]$ \footnote{This domain is the completion of $R'[S]$ at the maximal ideal generated by $\fp'$ and $S$; see, e.g., \cite[Exercise 7.11]{Eis95}.}, and denote its fraction field by $F$. Note that $F$ is contained in $k'_{\fp'}((S))$.

\begin{lemma} \label{lemma 1 step 2}
For every finite extension $M/k'(S)$, there exists a finite set $\mS_M$ of primes of $k'$ (depending only on $M/k'(S)$) such that, if $\fp'$ is not in $\mS_M$, then, the fields $M \cdot F$ and $k'_{\fp'}((S))$ are linearly disjoint over $F$.
\end{lemma}

\begin{proof}
Up to replacing $M/k'(S)$ by its Galois closure, we may assume that $M/k'(S)$ is Galois. Given an intermediate field $k'(S)\subseteq M_0\subseteq M$, let $\alpha:=\alpha(M_0)$ be a primitive element of $M_0/k'(S)$. We have 
\begin{equation} \label{alpha}
\alpha=\sum_{i \geq -n_0} \alpha_i \cdot S^{i/e}
\end{equation}
for some positive integer $e$, some non-negative integer $n_0$, and coefficients $\alpha_i$, $i\geq -n_0$, which lie in a finite extension $K/k'$. By possibly replacing $\alpha$ by $S^{n_0} \alpha$, we may assume $n_0=0$. As $\alpha$ is algebraic over $k'(S)$, Theorem \ref{thm:eisen} asserts that the set $\mathcal T_{M_0}$ of all primes $\fq'$ of $k'$ such that $v_{\fQ'}(\alpha_j) <0$ for some $j\geq 0$ and some prime $\fQ'$ of $K$ lying over $\fq'$ is finite, and depends only on $\alpha$, that is, only on $M_0$. 

Set $\mS_M := \cup_{M_0} \mathcal T_{M_0}$, where $M_0$ runs over all intermediate fields in $M/k'(S)$. Suppose that $\fp'$ is not in $\mS_M$, and set $F_1:=(M \cdot F)\cap k'_{\fp'}((S))$. To prove the lemma, it suffices to show that $F_1=F$. Since $F\subseteq F_1\subseteq M\cdot F$, and since the latter is Galois over $F$, the field $M_1:=F_1\cap M$ satisfies $M_1\cdot F = F_1$. Set  $\alpha:=\alpha(M_1)$. As $\alpha$ is an element of $k'_{\fp'}((S))$, one has $e=1$ (with the notation of \eqref{alpha}), so that $\alpha$ is an element of $R'_{\fp'}[[S]]$ (as $\fp'$ is not in $\mS_M$). Thus, $\alpha$, and hence $M_1$, are contained in $F$. Hence, $F_1$, which is equal to $M_1 \cdot F$, is equal to $F$, thus ending the proof of the lemma.
\end{proof}

Apply Lemma \ref{lemma 1 step 2} with $M=E' \cdot E'_{t_i}$ to get a finite set $\mS_M$ of primes of $k'$, depending only on $E/k(T)$, such that $(E' \cdot E'_{t_i}) \cdot F$ and $k'_{\fp'}((S))$ are linearly disjoint over $F$, provided $\fp'$ is not in $\mS_M$.
Set $L:= (E'_{t_i} \cdot k'_{\fp'}((S))) \cap (E' \cdot F)$. Since $(E'_{t_i}\cdot F)\cdot k'_\fp((S))\subseteq (E'\cdot F)\cdot k'_\fp((S))$, the previous linear disjointness provides $E'_{t_i}\cdot F\subseteq E'\cdot F$.  Hence, $E'_{t_i}\cdot F \subseteq L$. Conversely, $L$ contains $E'_{t_i}\cdot F$ by its definition and the inclusion $E'_{t_i}\cdot F\subseteq E'\cdot F$. Letting $\mS_2$ denote the finite set of primes of $k$ obtained by restricting a prime in $\mS_M$, the equality $L= E'_{t_i}\cdot F$ holds, provided $\fp \not\in \mS_2$. Set $d_{i,\fp'}:=[E'_{t_i}\cdot k'_{\fp'}:k'_{\fp'}]$. As $F$ contains $k'_{\fp'}=\Frac(R'_{\fp'})$, %one has
$$d_{i,\fp'}=[E'_{t_i}\cdot k'_{\fp'}:k'_{\fp'}]\geq [E'_{t_i}\cdot F:F] \geq [E'_{t_i}\cdot k'_{\fp'}((S)):k'_{\fp'}((S))] = d_{i,\fp'},$$
the last equality following from $E'_{t_i}\cdot k'_{\fp'}((S))/k'_{\fp'}((S))$ being unramified. Hence, $[E'_{t_i}\cdot F:F] = d_{i,\fp'}$. Moreover, as $E'\cdot F$ and $k'_{\fp'}((S))$ are linearly disjoint over $F$, the Galois group $\Gal(E'\cdot F/F)$ is identified with $V$ (=${\rm{Gal}}(E' \cdot k'_{\fp'}((S))/ k'_{\fp'}((S)))$; see \eqref{equ:diag1}) via restriction. This identification gives $\Gal(E'\cdot F/E'_{t_i}\cdot F) = I_{t_i}$, and $\Gal(E'_{t_i}\cdot F/F) = D_{t_i,\fp'}$. 

We then obtain the following diagram of inclusions and Galois groups:
\begin{equation}\label{equ:diag2}
\xymatrix{
E' \ar@{-}[rr] \ar@{-}[dd]& & E' \cdot F \ar@{-}[d]_{I_{t_i}} \ar@{-}[r] \ar@{-}@/^3pc/[dd]^>>>>>>>V &  E' \cdot k'_{\fp'}((S)) \ar@{-}[d]_{I_{t_i}} \ar@{-}@/^3pc/[dd]^>>>>>>>V \\
& E'_{t_i}\cdot k'_{\fp'} \ar@{-}[d]_{D_{t_i,\fp'}} \ar@{-}[r]&  L=E'_{t_i}\cdot F \ar@{-}[d]_{D_{t_i,\fp'}} \ar@{-}[r] &  E'_{t_i}\cdot k'_{\fp'}((S)) \ar@{-}[d]_{D_{t_i,\fp'}} \\
k'(S) \ar@{-}@/_1pc/[rr] & k'_{\fp'}  \ar@{-}[r] &  F  \ar@{-}[r] &  k'_{\fp'}((S)).
}
\end{equation}

\subsubsection{Step III}  \label{Step III}

Here, we reduce the extension $E'\cdot F /F$ modulo $(T-t_0)$, and identify the reduction with $E_{t_0}k_{\fp}/k_{\fp}$. 
More precisely, we show that there exists a place $\mathfrak{M}$ of $E'\cdot F$ whose restriction to $E$ has residue field $E_{t_0}$ and whose restriction to $F$ has residue field $k_\fp$. We then show that the residue field at $\mathfrak{M}$ is the compositum $E_{t_0} k_\fp$. 
%To make this reduction precise, we define places of $E'\cdot F$ and $F$ lying over $(T-t_0)$ of $k(T)$. 

As $t_0$ and $t_i$ meet modulo $\fp'$ by the definition of $\fp'$, and as $t_i$ is integral over $R$, one has $v_{\fp'}(t_0-t_i)>0$. Hence, $T-t_0 = S-(t_0-t_i)$ is in $R'_{\fp'}[[S]]$. Moreover, as $((t_0-t_i)^m)_{m \geq 1}$ converges to $0$ in $R'_{\fp'}$, the specialization map $R'_{\fp'}[[S]]\ra R'_{\fp'}$, which sends $S$ to $t_0-t_i$, is well-defined. As it is onto, there is a canonical isomorphism $R'_{\fp'}[[S]]/(T-t_0) \, R'_{\fp'}[[S]]\cong  R'_{\fp'}$. In particular, $(T-t_0) \, R'_{\fp'}[[S]]$ is a prime ideal of $R'_{\fp'}[[S]]$. Let $\mathfrak{R}$ be the localization of $R'_{\fp'}[[S]]$ at $(T-t_0)\, R'_{\fp'}[[S]]$. The previous isomorphism shows that the residue field of $\mathfrak{R}$ at $(T-t_0)\, \mathfrak{R}$ is canonically isomorphic to ${\rm{Frac}}(R'_{\fp'}) = k'_{\fp'}$. Let $\mathcal{S}_3$ be the finite set of primes of $k$ which ramify in $k'/k$. Assuming $\fp \not \in \mathcal{S}_3$, and using that $f_{\fp' | \fp}=1$ by the definition of $\fp'$, we get $k'_{\fp'} = k_\fp$. Hence, $\mathfrak{R}/(T-t_0)\, \mathfrak{R}$ is canonically isomorphic to $k_\fp$. We use this canonical isomorphism to identify the two fields. 
%From now on we shall identify  $\mathfrak{R}/(T-t_0)\, \mathfrak{R}$ with $k_\fp$ via this canonical isomorphism. 

The integral closure $\mathfrak{S}$ of $\mathfrak{R}$ in $E'\cdot F$ is a Dedekind domain containing a prime ideal $\mathfrak{M}$ lying over $(T-t_0) \, \mathfrak{R}$. In particular, $\mathfrak{M}\cap E$ is a prime of $E$ lying over $(T-t_0) \, k[T]$. Let $\mathfrak R_{t_i}$ (resp., $\mathfrak{M}_{t_i}$) be the restriction to $E'_{t_i}\cdot F$ of $\mathfrak{S}$ (resp., of $\mathfrak{M}$). Since $E'_{t_i}$ and $k'_{\fp'}$ are contained in the residue field $\mathfrak{R}_{t_i}/\mathfrak{M}_{t_i}$, and since $[E'_{t_i}\cdot k'_{\fp'}:k'_{\fp'}]=d_{i,\fp'} = [E'_{t_i}\cdot F:F]$, we have $\mathfrak{R}_{t_i}/\mathfrak{M}_{t_i}=E'_{t_i}\cdot k'_{\fp'}$. Thus, we have the following diagram of inclusions:

\begin{equation}\label{equ:residue}
\xymatrix{
E_{t_0} = \oline E_{\mathfrak{M}\cap E} \ar@{-}[r] \ar@{-}[dd] & \mathfrak{S}/\mathfrak{M} \ar@{-}[d] \\%= k_{\fp}(\oline y_1,\ldots,\oline y_n) \\ %\ar@{-}[r]^{\psi}   &  k'_{\fp'}(z_1,\ldots, z_n) \ar@{-}[d] \\
& \mathfrak{R}_{t_i}/\mathfrak{M}_{t_i}= E'_{t_i}\cdot k'_{\fp'} \ar@{-}[d] \\
k \ar@{-}[r] & \mathfrak{R}/(T-t_0) \, \mathfrak{R} = k_{\fp}=k'_{\fp'}.  
}
\end{equation}

Now, we claim that $\mathfrak{S}/\mathfrak{M}$ is equal to the compositum of $k_\fp$ and $E_{t_0}$, provided $\fp$ is not in some finite set of primes of $k$ defined below. Indeed, let $P(T,Y) \in R[T][Y]$ be the minimal polynomial over $k(T)$ of some primitive element of $E/k(T)$, assumed to be integral over $R[T]$. If $P(t_0,Y)$ is not separable, then, $t_0$ belongs to the finite set $D$ of all roots in $k$ of the discriminant $\Delta(T) \in R[T]$ of $P(T,Y)$ which are not branch points of $E/k(T)$. As $t_0$ and $t_i$ meet modulo $\fp$, the inseparability of $P(t_0,Y)$ implies that $\fp$ is in the finite set $\mS_4$ of primes $\fq$ of $k$ such that $v_{\fq'}(d-t_i)>0$ for some $d\in D$ and some prime $\fq'$ of $k'$ lying over $\fq$. Henceforth, we shall assume that $\fp$ is not in $\mS_4$, and then that $P(t_0,Y)$ is separable. Denote the roots of $P(T,Y)$ in $E$ by $y_1,\ldots,y_n$, and their reductions modulo $\mathfrak{M}\cap E$  by $\oline y_1,\dots,\oline y_n$, respectively. As $\Delta(t_0) \not=0$, the discriminant $\Delta(T)$ is not in $(T-t_0) \, \mathfrak{R}$. Since $E'\cdot F$ is generated by $y_1,\ldots,y_n$ over $F$, the residue field $\mathfrak{S}/\mathfrak{M}$ is generated by $\oline y_1,\ldots,\oline y_n$ over $\mathfrak{R}/(T-t_0) \, \mathfrak{R}=k_\fp$ by Lemma \ref{dedekind}. Thus, one has $\mathfrak{S}/\mathfrak{M} = E_{t_0}\cdot k_\fp$, proving the claim. The general containment in part (2) of Theorem \ref{thm main} follows from \eqref{equ:residue}, provided $\fp$ is not in $\mS_1\cup\mS_2\cup\mS_3\cup\mS_4$. 

\subsubsection{Step IV} 
%We describe the decomposition group of $E'\cdot F/F$ at a prime lying over $(T-t_0)$ both as a subgroup of $V$ and as a subgroup of $D_{t_0}$. 
Here, we use Step III to identify the Galois group of $E_{t_0}k_{\fp}/k_{\fp}$ with a subgroup of $\Gal(E'\cdot F/F)=\varphi^{-1}(D_{t_i,\fp'})$ whose projection under $\varphi$ equals $D_{t_i,\fp'}$. 

Assume that $\fp\not\in\mS_1\cup\mS_2\cup\mS_3\cup\mS_4$, so that we are in the situation of \eqref{equ:diag2} and \eqref{equ:residue}. Up to replacing $D_{t_0}$ by a conjugate of it, we may assume that $D_{t_0}$ is the decomposition group of $E/k(T)$ at $\mathfrak{M}\cap E$. Let $U$ be the decomposition group of $E'\cdot F/F$ at $\mathfrak{M}$, so that it identifies via restriction with a subgroup of $D_{t_0}$. As the prime of $k[T]$ generated by $T-t_0$ is unramified in $E/k(T)$, it is also unramified in $E'\cdot F/F$. Thus, $U$ (resp., $D_{t_0}$) is also the Galois group of $E_{t_0}\cdot k_\fp/k_\fp$ (resp., of $E_{t_0}/k$). Hence, the restriction of $U$ to $E_{t_0}/k$ is the decomposition group of some prime of $E_{t_0}$ lying over $\fp$. Thus, we may identify $U$ with $D_{t_0,\fp}$. 

On the other hand, $U$ is a subgroup of $\Gal(E'\cdot F/F)$. The latter is identified with $V$ in Step II (\S\ref{Step II}), and hence is a subgroup of $D_{t_i}$ via the identification in Step I (\S\ref{Step I}). Moreover, $\varphi(U)$ is the decomposition group of  $E'_{t_i}\cdot F/F$ at $\mathfrak{M}_{t_i}$. As shown in Step III (\S\ref{Step III}), one has $\Gal(E'_{t_i}\cdot F/F) = {\rm{Gal}}(E'_{t_i} \cdot k_{\fp'} / k_{\fp'})$. But the latter is $D_{t_i,\fp'}$ by Step II. Hence, $\varphi(U) = D_{t_i,\fp'}$, thus proving the general case of part (1) of Theorem \ref{thm main}. 

Finally, assume that $I_\fp(t_0,t_i)$ is coprime to $e_i$. Let $I_\mathfrak{M}$ denote the inertia group of $E'\cdot F/F$ at $\mathfrak{M}$. Since $E'_{t_i}\cdot F/F$ is unramified at $\mathfrak{M}_{t_i}$, one has $I_\mathfrak{M}\subseteq {\rm{Gal}}(E' \cdot F /E'_{t_i} \cdot F)$ ($=I_{t_i}$; see \S\ref{Step II}). Set $\mathcal{S}_{\rm{exc}} := \mS_1\cup\mS_2\cup\mS_3\cup\mS_4 \cup \mathcal{S}_{\rm{bad}}$, where $\mathcal{S}_{\rm{bad}}$ is the finite set of primes of $k$ from the Specialization Inertia Theorem (\S\ref{sec:basics_inertia}), and assume that $\fp$ is not in $\mathcal{S}_{\rm{exc}}$. As $I_\fp(t_0,t_i)$ is coprime to $e_i$, one has $|I_{t_0,\fp}|=e_i$ by the Specialization Inertia Theorem. Since the conjugation of $U$ to $D_{t_0,\fp}$ sends $I_\mathfrak{M}$ to $I_{t_0,\fp}$, we also have $|I_\mathfrak{M}|=e_i$. As $I_\mathfrak{M}$ is a subgroup of $I_{t_i}$,  and since $|I_{t_i}|=e_i$, we get the equality $I_\mathfrak{M} = I_{t_i}$. Thus, $I_{t_i}$ is contained in $U$, and $\mathfrak{M}_{t_i}$ is totally ramified in $E'\cdot F/E'_{t_i}\cdot F$. Hence, $U=\varphi^{-1}(D_{t_i,\fp'})$, and $(E_{t_0}\cdot k_\fp)^{\rm{ur}} = \mathfrak{R}_{t_i}/\mathfrak{M}_{t_i} = E'_{t_i}\cdot k'_{\fp'}$, completing the proof of Theorem \ref{thm main}.

\subsection{On specifying completions in specializations} \label{sec:main_applications}

Theorem \ref{thm main} restricts the structure of the completion $(E_{t_0})_\fp$ at $\fp$ of $E_{t_0}$. Namely, it implies that  $(E_{t_0})_\fp$ contains the field $N^{(\fp)}:=(E(t_i))_{t_i}\cdot k(t_i)_{\fp'}$ whose Galois group over $k(t_i)_{\fp'}$  %$\Gal(N^{(\fp)}/k(t_i)_{\fp'})$ 
is $D_{t_i,\fp'}$, and that the Galois group $\Gal((E_{t_0})_\fp/k_\fp)$ is a subgroup of $\varphi^{-1}(D_{t_i,\fp'})$, where $\varphi:D_{t_i}\ra D_{t_i}/I_{t_i}$ is the natural projection. The following theorem shows that this is the only restriction for extensions of $k_\fp$ with Galois group $\varphi^{-1}(D_{t_i,\fp'})$: 
%Since Theorem \ref{thm main} restricts the structure of completions at $\fp$ of specializations of $E/k(T)$, it is natural to ask which of the finite Galois extensions of $k_\fp$, compatible with these restrictions, can occur as completions at $\fp$ of some suitable specializations of $E/k(T)$. Theorem \ref{thm fields} below shows that all of them with maximal possible decomposition group occur.

\begin{theorem} \label{thm fields}
Let $\mathcal{S}$ be a finite set of primes of $k$, disjoint from some finite set $\mathcal{S}'_{\rm{exc}}= \mS'_{\rm{exc}}(E/k(T))$. For each $\fp\in \mathcal{S}$, assume that there exists $i:=i(\fp)\in \{1, \dots ,r\}$ and a prime $\fp'$ lying over $\fp$ in $k(t_i)/k$ with residue degree $1$, and let $L^{(\fp)}/k_\fp$ be a finite Galois extension containing $N^{(\fp)}:=(E(t_i))_{t_i}\cdot k(t_i)_{\fp'}$ such that there exists an isomorphism
$\psi$ from ${\rm{Gal}}(L^{(\fp)}/k_\fp)$ to $\varphi^{-1}(D_{t_i,\fp'})$ which maps ${\rm{Gal}}(L^{(\fp)}/N^{(\fp)})$ onto $I_{t_i}$.
Then, there exist infinitely many $t_0\in k$ such that the completion of $E_{t_0}/k$ at $\fp$ equals $L^{(\fp)}/k_\fp$ for each $\fp\in \mathcal{S}$. Moreover, if $k$ is a number field,  we may assume that these specializations of $E/k(T)$ have Galois group $G$.
\end{theorem}

\begin{remark}
(1) If $L^{(\fp)}/N^{(\fp)}$ is totally ramified, $\fp \in \mS$, and if $k$ is a number field, then, we are in the settings of Theorem \ref{thm:fields}. In this case, by the Specialization Inertia Theorem, the intersection multiplicity of $t_0$ and $t_{i}$ at $\fp$ is coprime to $|I_{t_{i}}|$, $\fp \in \mS$. \\
(2) By combining Theorem \ref{thm fields}, \cite[Theorem 1.2]{DG12}, and the Chinese remainder theorem, one can more generally require the extensions $L^{(\fp)}/k_\fp$, $\fp \in \mS$, to be either of the above form or unramified of degree $d$, where $d$ is an arbitrary positive integer such that $G$ contains at least one element of order $d$.
\end{remark}

\begin{proof}
Fix a prime $\fp\in \mathcal{S}$. As in \S\ref{sec:main_proof}, set $k':=k(t_i)$, $k'_{\fp'} :=k(t_i)_{\fp'}$, $E':=E(t_i)$, $E'_{t_i}:=(E(t_i))_{t_i}$, and $S:=T-t_i$. Moreover, set $N:=N^{(\fp)}$, and $\Gamma:=\varphi^{-1}(D_{t_i,\fp'})$. As shown in Step I of the proof of Theorem \ref{thm main} (\S\ref{Step I}), one has ${\rm {Gal}}(N/k_\fp)=D_{t_i,\fp'}$. Then, all fields $L^{(\fp)}$ with the properties mentioned in Theorem \ref{thm fields} are solution fields to the embedding problem 
%induced by $1\to I_{t_i}\to 
$\Gamma \stackrel{\varphi}{\to} {\rm {Gal}}(N/k_\fp)$. % where $\Gal(N/k_\fp)$ is idenfitied with $D_{t_i}/I_{t_i}$. % \to 1.$
Since $I_{t_i}$ (resp., $\Gamma$) is the inertia group (resp., the Galois group) of $E'\cdot k'_{\fp'}((S)) / k'_{\fp'}((S))$ (as shown in Step I of the proof of Theorem \ref{thm main} (\S\ref{Step I})), the action of $\Gamma$ on the cyclic kernel $I_{t_i}$ is isomorphic to the action on $e_i$-th roots of unity in $N$ (see \S\ref{sec:basics_laurent}). The above embedding problem is then a Brauer embedding problem. By \cite[Chapter IV, Theorem 7.2]{MM99}, if $x\in N$ is chosen such that $N(\sqrt[e_i]{x})$ is a solution field to this embedding problem, then, all the solutions fields are of the form $N(\sqrt[e_i]{\beta x})$ with $\beta \in k_\fp^\times$. Furthermore, upon multiplying $\beta$ by a suitable $e_i$-th power, we can require $\beta x$ to be of positive $\fp$-adic valuation.

The field $E'\cdot k'((S))$ is generated over $E'_{t_i}\cdot k'((S))$ by $\sqrt[e_i]{\alpha S}$ for some $\alpha \in E'_{t_i}$; see \S\ref{sec:basics_laurent}. Up to enlarging the set $\mathcal{S}'_{\rm{exc}}$, we may assume that $\alpha$ is of $\fp$-adic valuation $0$. Set $M:=(E'_{t_i}\cdot k'(S))(\sqrt[e_i]{\alpha S})$, and consider the field $F$ from Step II of the proof of Theorem \ref{thm main} (\S\ref{Step II}). Assume that $\mathcal{S}'_{\rm{exc}}$ contains the set $\mathcal{S}_{\rm{exc}}$ from Theorem \ref{thm main}. Then, $(E'\cdot M)\cdot F$ and $k'_{\fp'}((S))$ are linearly disjoint over $F$ by Lemma \ref{lemma 1 step 2}. But one also has $(M\cdot F) \cdot k'_{\fp'}((S)) = E'\cdot k'_{\fp'}((S))$. Hence, $M\cdot F = E'\cdot F$.

Thus, we have $E'\cdot F = (E'_{t_i}\cdot F)(\sqrt[e_i]{\alpha S})$. Next, we choose $t_0\in k_\fp$ with $I_\fp(t_0,t_i)>0$ and specialize $E'\cdot F/F$ at $T-t_0$ as in Step III of the proof of Theorem \ref{thm main} (\S\ref{Step III}). Since $E'_{t_i}\cdot F$ specializes to $E'_{t_i} \cdot k'_{\fp'}$ and $\sqrt[e_i]{\alpha S}$ specializes to $\sqrt[e_i]{\alpha(t_0-t_i)}$, Lemma \ref{dedekind} yields that all fields of the form $(E'_{t_i} \cdot k'_{\fp'})(\sqrt[e_i]{\alpha \pi})$ ($=N(\sqrt[e_i]{\alpha \pi})$), with $\pi$ an element of $k_\fp=k'_{\fp'}$ of positive $\fp$-adic valuation, can be reached via specializing $T$ to $t_0\in k_\fp$. Hence, if $N(\sqrt[e_i]{x})$ is one of the fields obtained in this way, then, we reach all fields of the form $N(\sqrt[e_i]{\beta x})$, where $\beta$ is any element of $k_\fp^\times$ such that $\beta x$ is of positive $\fp$-adic valuation. As shown above, this covers in particular the extension $L^{(\fp)}/k_\fp$.
Krasner's lemma shows that $L^{(\fp)}/k_\fp$ can even be obtained by restricting to specialization values $t_\fp\in k$.
The extension $L^{(\fp)}/k_\fp$ therefore equals $(E_{t_\fp})_\fp/k_\fp$, the latter being the specialization of $E'_{t_i}\cdot F/F$ at $T-t_\fp$. % for suitable $t_0=:t_0(\fp)\in k$. 

Finally, choose $t_0\in k$ sufficiently close $\fp$-adically to every $t_\fp\in \mathcal{S}$. This yields the assertion in the general case. In the  number field case, Lemma \ref{PV} provides the conclusion on the Galois group of the produced specializations.
\end{proof}

\subsection{On specifying inertia and decomposition groups of specializations} \label{sec:main_chebotarev}

Our next result is devoted to a re-occurrence property for subgroups of $G$ appearing as decomposition groups of specializations, in the case where $k$ is a number field.

\begin{theorem} \label{chebotarev}
Assume that $k$ is a number field, and let $s$ be a positive integer. Given $j \in \{1,\dots,s\}$, choose a branch point $t_{i(j)}$ of $E/k(T)$, and an element $\tau_{i(j)}$ of $D_{t_{i(j)}}$. 
Then, there exist $s$ infinite sets $\mathcal{S}_1, \dots, \mathcal{S}_s$ of primes of $k$ which satisfy the following property. For every tuple $(\fp_1, \dots, \fp_s) \in \mathcal{S}_1 \times \dots \times \mathcal{S}_s$ 
%such that $\fp_1, \dots, \fp_s$ are distinct, 
of distinct primes, there exist infinitely many $t_0$ in $k$ for which $E_{t_0}/k$ is a $G$-extension whose decomposition group (resp., inertia group)  
at $\fp_j$ is $\langle I_{t_{i(j)}}, \tau_{i(j)} \rangle$ (resp., $I_{t_{i(j)}}$) for $j=1,\ldots,s$. 
%such that the following two conditions hold:
%\begin{itemize}
%\item[(1)] $\Gal(E_{t_0}/k) = G$,
%\item[(2)] for each $j \in \{1,\dots,s\}$, the inertia group (resp., the decomposition group) of $E_{t_0}/k$ at $\fp_j$ is equal to $I_{t_{i(j)}}$ (resp., to $\langle I_{t_{i(j)}}, \tau_{i(j)} \rangle$).
%\end{itemize}
\end{theorem}

\begin{proof}
We may assume without loss that $s=1$. Indeed, Lemma \ref{PV} yields that the assertion for one prime $\fp_i$ at a time ($i=1,\dots,s$) implies the assertion for all primes $\fp_1,\dots,\fp_s$ simultaneously. 
Also, by that lemma, we do need to prove that $E_{t_0}/k$ has Galois group $G$. As before, set $i:=i(1)$, $t_i:=t_{i(1)}$, $\tau := \tau_{i(1)}$, $k':=k(t_i)$, and $E'_{t_i}:= (E(t_i))_{t_i}$. Let $\overline{\tau}$ be the image of $\tau$ under the natural projection $D_{t_i} \rightarrow D_{t_i}/I_{t_i}$, and let $C$ be the conjugacy class in $D_{t_i}/I_{t_i}$ of $\overline{\tau}$. By part (1) of Theorem \ref{thm main} and part (1) of Remark \ref{rk thm main}, it suffices to show that there is an infinite set $\mathcal{S}$ of primes of $k$ such that, for every $\fp \in \mathcal{S}$, there is a prime $\fp'$ lying over $\fp$ in $k'/k$, with residue degree $f_{\fp' | \fp}=1$, and such that the {\textit{Frobenius}} ${\rm{Frob}}_{\fp'}(E'_{t_i} / k')$ lies in $C$. By the Chebotarev density theorem, the natural density of the set $\mathcal{S'}$ of all primes $\fp'$ of $k'$ such that ${\rm{Frob}}_{\fp'} (E'_{t_i} / k')$ lies in $C$ equals $|C| \cdot |I_{t_i}| / |D_{t_i}|$. This remains true for the set $\mathcal{S}''=\{\fp'\in \mathcal{S}' : f_{\fp' | \fp' \cap k}=1\}$, since the set of residue degree $1$ primes $\fp'$ of $k'$ is of natural density $1$,
 by the prime number theorem for number fields. % due to Landau \cite{Lan03}. 
 In particular, the set of all primes $\fp$ of $k$ which are contained in a prime $\fp' \in \mathcal{S}''$ is infinite, completing the proof.
\end{proof}

\section{Application to $G$-crossed products and admissibility} \label{sec:admiss}

This section is devoted to our application to $G$-crossed products over number fields. We state and prove our admissibility criterion, Theorem \ref{admiss_crit}, in \S\ref{sec:admiss_crit}, and apply it to obtain explicit families of examples in \S\ref{sec:admiss_examples}. %, as already mentioned in \S\ref{sec:intro_admiss}. 
%After recalling some background in \S\ref{background} (completing that from \S\ref{sec:intro_admiss}), we state and prove Theorem \ref{admiss_crit}, our new general admissibility criterion, in \S\ref{sec:admiss_crit}. Explicit examples, including those given in Theorem \ref{thm:admis} from the introduction, are then given in \S\ref{sec:admiss_examples}.
For this section, let $k$ be a number field, $R$ its ring of integers, $T$ an indeterminate over $k$, and $G$ a finite group.

\subsection{Background} \label{background}
%Let us recall the following  terminology:
%\begin{definition}\label{def:admiss} \
%\begin{itemize}
%\item[(1)] 
Recall that $G$ is called {\textit{$k$-admissible}} if there exists a $G$-crossed product division algebra with center $k$.
%\item[(2)] 
A $G$-extension $L/k$ such that $L$ is a maximal subfield of a $G$-crossed product division algebra is called {\textit{$k$-adequate.}}
%\end{itemize}
%\end{definition}
To prove Theorem \ref{admiss_crit}, we need Schacher's admissibility criterion  \cite[\S2]{Sch68} over number fields:
%, due to Schacher \cite[\S2]{Sch68}:

\begin{theorem} \label{schacher}
Let $L/k$ be a $G$-extension. Then, $L$ is $k$-adequate if and only if, for each prime number $p$ dividing $|G|$, there exist two distinct prime ideals $\fp_1$ and $\fp_2$ of $R$ such that the decomposition group of $L/k$ at $\fp_i$ contains a $p$-Sylow subgroup of $G$ for $i =1,2$.
\end{theorem}

\subsection{A new general admissibility criterion} \label{sec:admiss_crit}

\begin{theorem} \label{admiss_crit}
Assume that $G$ has a $k$-regular realization $E/k(T)$ such that, for every prime number $p$ such that $G$ has a non-cyclic $p$-Sylow subgroup, the following holds:

\vspace{1mm}

\noindent
{\rm{(H)}} for some branch point $t_i$ of $E/k(T)$, the decomposition group $D_{t_i}$ contains a $p$-Sylow subgroup $P$ of $G$ such that $PI_{t_i}/I_{t_i}$ is cyclic.

\vspace{1mm}

\noindent
Then, there exist infinitely many pairwise linearly disjoint $k$-adequate $G$-extensions of $k$, arising as specializations of $E/k(T)$. In particular, $G$ is $k$-admissible.
\end{theorem}

\begin{proof}
%We verify Schacher's criterion for each $p$ dividing $|G|$.
Let $p$ be a prime number dividing $|G|$, and let $P$ be a $p$-Sylow subgroup of $G$. First, assume that $P$ is cyclic. Let $t_p \in k$ be such that $E_{t_p}/k$ has Galois group $G$; such a $t_p$ exists by Hilbert's irreducibility theorem. By the Chebotarev density theorem, there exist infinitely many prime ideals $\fp$ of $R$ such that $D_{t_p,\fp}$ contains $P$. Now, assume that $P$ is not cyclic. Let $t_i(p)$ be a branch point of $E/k(T)$ such that the decomposition group $D_{t_i(p)}$ contains $P$, and such that $PI_{t_i(p)}/I_{t_i(p)}$ is cyclic (condition (H)). Then, by Theorem \ref{chebotarev}, there exists an infinite set $\mathcal{S}_p$ of prime ideals of $R$ such that, for each $\fp \in \mathcal{S}_p$, there exist infinitely many $t_{p} \in k$ such that $D_{t_p,\fp}=PI_{t_i(p)}$. 
By applying Lemma \ref{PV}, we obtain $t_0\in k$ such that $E_{t_0}/k$ is a $G$-extension and, for each $p$ dividing $|G|$,  one has $D_{t_0,\fp}=PI_{t_i(p)}$ for at least two distinct primes $\fp$ of $k$, verifying Theorem \ref{schacher}. % and a $p$-Sylow subgroup $P$. 
%It then remains to use Lemma \ref{PV} and Theorem \ref{schacher} to conclude.
\end{proof}

\subsection{Examples}\label{sec:admiss_examples}

Roughly speaking, in order to apply Theorem \ref{admiss_crit}, one has to produce $k$-regular realizations of $G$ with ``sufficiently large" ramification indices and residue degrees at some branch points. Ramification indices can often be prescribed purely group-theoretically (e.g., via the {\textit{rigidity method}}). To ensure ``large" residue degrees we use Lemma \ref{branchcycle}. 
% it is in general difficult to predict the residue degrees without an explicit defining polynomial. However, special cases can easily be obtained, by using Lemma \ref{branchcycle}. 
This method yields a large class of new examples.
Below, we derive the first $\qq$-admissibility results for infinite families of finite non-abelian simple groups. More precisely, Theorem \ref{thm:psl1} covers half of the groups ${\rm{PSL}}_2(\mathbb{F}_p)$ ($p$ prime), whereas Theorem \ref{thm:psl2} covers a further $1/8$, giving a total of $62.5\%$ of the groups. 

\begin{theorem} \label{thm:psl1}
Assume that $k$ does not contain $\sqrt{-1}$. Let $p$ be a prime number such that either $p \equiv 3 \pmod 8$ or $p \equiv 5 \pmod 8$. Then, the groups ${\rm{PSL}}_2(\mathbb{F}_p)$ and ${\rm{PGL}}_2(\mathbb{F}_p)$ are $k$-admissible.
\end{theorem}

\begin{proof}
Recall that ${\rm{PGL}}_2(\mathbb{F}_p)$ has order $(p-1)p(p+1)$. Then, by our assumption on $p$, the $2$-Sylow subgroups of ${\rm{PGL}}_2(\mathbb{F}_p)$ (resp., of ${\rm{PSL}}_2(\mathbb{F}_p)$) have order $8$ (resp., 4). By \cite[Chapter I, Corollary 8.10]{MM99}, there exists a $k$-regular ${\rm{PGL}}_2(\mathbb{F}_p)$-extension $E/k(T)$ with $k$-rational branch points $t_1$, $t_2$, and $t_3$, and such that the corresponding inertia groups $I_{t_1}$, $I_{t_2}$, and $I_{t_3}$ are generated by elements in the conjugacy classes  $2B$, $4A$, and $pA$ of ${\rm{PGL}}_2(\mathbb{F}_p)$, respectively (according to the ATLAS notation \cite{ATL}). Consider the branch point $t_2$ of $E/k(T)$ with ramification index 4. By Lemma \ref{branchcycle}, and as $k$ does not contain $\sqrt{-1}$, the residue degree at this branch point is divisible by 2. It then remains to combine Theorem \ref{admiss_crit} and the classical fact that the Sylow subgroups of ${\rm{PGL}}_2(\mathbb{F}_p)$ of odd order are cyclic to get that ${\rm{PGL}}_2(\mathbb{F}_p)$ is $k$-admissible.

Furthermore, the fixed field of ${\rm{PSL}}_2(\mathbb{F}_p)$ in $E$ is a rational function field, say $k(S)$, by, e.g., \cite[Lemma 4.5.1]{Ser92}, that has degree 2 over $k(T)$. Denote by $s_2$ the unique point lying over $t_2$ in $k(S)/k(T)$. Then, $s_2$ is a branch point of $E/k(S)$ with ramification index 2, and the residue degree at $s_2$ in $E/k(S)$ is divisible by 2. As above, we conclude that ${\rm{PSL}}_2(\mathbb{F}_p)$ is $k$-admissible.
\end{proof}

\begin{theorem} \label{thm:psl2}
 Let $p$ be a prime number such that the following two conditions hold:
\begin{itemize}
\item[(1)] either $p \equiv 7 \pmod {16}$ or $p \equiv 9 \pmod {16}$,
\item[(2)] either $p \equiv 2 \pmod {5}$ or $p \equiv 3 \pmod {5}$.
\end{itemize}
Assume that $k$ contains neither $\sqrt{-1}$ nor $\sqrt{p^\star}$, where $p^\star = (-1)^{(p-1)/2}p$. Then, the group ${\rm{PSL}}_2(\mathbb{F}_p)$ is $k$-admissible.
\end{theorem}

\begin{proof}
By condition (1), the $2$-Sylow subgroups of ${\rm{PSL}}_2(\mathbb{F}_p)$ have order $8$. Moreover, $5$ is not a square modulo $p$ by condition (2). By \cite[Chapter I, Theorem~7.10]{MM99}, one then has a $\Qq$-regular ${\rm{PSL}}_2(\mathbb{F}_p)$-extension of $\Qq(T)$ with 4 branch points, and the corresponding inertia groups are generated by elements in the conjugacy classes $4A$, $4A$, $pA$, and $pB$ of ${\rm{PSL}}_2(\mathbb{F}_p)$ (according to the ATLAS notation). The proof of that theorem also gives the position of the branch points of this $\Qq$-regular extension. Namely, in the notation of that proof, the $\Qq$-regular realization of ${\rm{PSL}}_2(\mathbb{F}_p)$ is $E/\qq(v)$, and the branch points with ramification index 4 are the roots of the equation $u^2-6u+25=0$, where $u$ is related to $v$ via $v:=\sqrt{p^\star}(u+5)/({u-5})$. From the above data, one computes that the fields $\qq(v_i)$, where $v_i$ is a branch point of ramification index 4, do not contain $\sqrt{-1}$. In fact, one has $v_i=\pm 2\sqrt{-p^\star}$. Clearly, all of the above remains true for the extension $E \cdot k/k(v)$ instead of $E/\qq(v)$. In particular, as $k$ contains neither $\sqrt{-1}$ nor $\sqrt{p^\star}$, one has $\sqrt{-1} \not \in k(v_i)$. By Lemma \ref{branchcycle}, the residue degree at a branch point of $E \cdot k/k(v)$ with ramification index 4 is divisible by 2. Hence, the decomposition group at such a branch point is of order divisible by $8$. It then remains to apply Theorem \ref{admiss_crit} to conclude.
\end{proof}

\section{Application to the Hilbert-Grunwald property} \label{sec:grunwald}

The present section is devoted to our application to the Hilbert-Grunwald property. % as already mentioned in \S\ref{sec:intro_results_Hilbert-Grunwald}. First, 
We state and prove Theorem \ref{thm:neggrunwald}, which is a more general version of Theorem \ref{neg_grunwald_boohoo}, and show that it applies to any finite non-abelian simple group; see Corollary \ref{simple}. %from the introduction (\S\ref{sec:intro_results_Hilbert-Grunwald}). Next, we show that Theorem \ref{thm:neggrunwald} can be applied to any finite non-abelian simple group; see Corollary \ref{simple}.
For this section, let $k$ be a number field, $R$ its ring of integers, $T$ an indeterminate over $k$, and $G$ a finite group.
Given finitely many $k$-regular $G$-extensions $E_1/k(T)$, $\dots$, $E_s/k(T)$, we are interested in the following question, for which the partial positive answer provided by Theorem \ref{thm fields} is a natural motivation:

\begin{question} \label{question}
Does there exist a finite set $\mathcal{S}_{\rm{exc}}$ of non-zero prime ideals of $R$ such that every Grunwald problem $(G, (L^{(\fp)}/k_\fp)_{\fp\in \mathcal{S}})$, with $\mathcal{S}$ disjoint from $\mathcal{S}_{\rm{exc}}$, has a solution inside the union of the sets of specializations of $E_1/k(T)$, $\dots$, $E_s/k(T)$?
\end{question}

It follows from earlier works by the first two authors \cite{Leg16a, Koe17a} that, for each non-trivial finite group $G$ for which there exists a $k$-regular $G$-extension of $k(T)$, there always is one of these realizations for which the answer to Question \ref{question} is negative.

In Theorem \ref{thm:neggrunwald} below, we show that the answer to Question \ref{question} is in fact {\textit{always}} negative for many finite groups:

\begin{theorem} \label{thm:neggrunwald}
Assume that $G$ has a non-cyclic abelian subgroup. Then, the answer to Question \ref{question} is negative for all finite sets of $k$-regular $G$-extensions of $k(T)$.
\end{theorem}

The proof of Theorem \ref{thm:neggrunwald} requires the following auxiliary result, which strengthens the special case of Theorem \ref{chebotarev} where $\tau_{i(j)}$ is chosen to be the identity element of $D_{t_{i(j)}}$. 

\begin{proposition} \label{inf_many_cyclic}
Let $E/k(T)$ be a $k$-regular $G$-extension, and $F$ the compositum of the residue fields $(E(t_1))_{t_1}, \dots, (E(t_r))_{t_r}$ of $E/k(T)$ at the branch points $t_1,\dots,t_r$. Moreover, let $\fp$ be a prime ideal of $R$ that is totally split in $F/k$ (avoiding a finite set of primes depending only on $E/k(T)$). Then, the decomposition group of $E_{t_0}/k$ at $\fp$ is cyclic for every $t_0 \in \mathbb{P}^1(k)$.
\end{proposition}

\begin{proof}
Let $t_0$ be in $\mathbb{P}^1(k)$. First, assume that $t_0$ is in $\{t_1,\dots,t_r\}$. Then, by the definition of $F$, the decomposition group $D_{t_0,\fp}$ of $E_{t_0}/k$ at $\fp$ is trivial. Now, assume that $t_0$ is not in $\{t_1,\dots,t_r\}$. If $t_0$ does not meet any branch point of $E/k(T)$ modulo $\fp$, then, by the Specialization Inertia Theorem (\S\ref{sec:basics_inertia}), $\fp$ is unramified in $E_{t_0}/k$ (up to excluding finitely many primes of $k$). Hence, $D_{t_0,\fp}$ is cyclic. So we may assume that $t_0$ meets some branch point $t_i$ modulo $\fp$. As $\fp$ is totally split in $F/k$, part (1) of Theorem \ref{thm main} shows that $D_{t_0,\fp}$ is contained in the inertia group $I_{t_i}$ of $E(t_i)/k(t_i)(T)$ at $(T-t_i) \, k(t_i)[T-t_i]$ (up to excluding finitely many primes of $k$). As $I_{t_i}$ is cyclic, we are done.
\end{proof}

\begin{comment}
\begin{remark} \label{rk_inf_many_cyclic}
If finitely many extensions $E_1/k(T), \dots, E_s/k(T)$ are given, each with a suitable number field $F_i$ as in Proposition \ref{inf_many_cyclic}, then, the conclusion clearly holds simultaneously for (the specializations of) $E_1/k(T), \dots, E_s/k(T)$, if $\fp$ is totally split in the composite extension $F_1\cdots F_s/k$ (up to excluding finitely many primes of $k$).
\end{remark}
\end{comment}

\begin{proof}[Proof of Theorem \ref{thm:neggrunwald}]
Given a positive integer $s$, let $\{E_1/k(T), \dots, E_s/k(T)$\} be a finite set of $k$-regular $G$-extensions. Let $F$ be the compositum of the residue extensions at branch points of $E_1/k(T), \dots, E_s/k(T)$. Then, by Proposition \ref{inf_many_cyclic}, for every non-zero prime ideal $\fp$ of $R$ that is totally split in $F/k$ (outside some finite set $\mathcal{S}_{{\rm{exc}}}$ depending only on $E_1/k(T),\dots,E_s/k(T)$), the completion at $\fp$ of every specialization of any of the extensions $E_1/k(T), \dots, E_s/k(T)$ has cyclic Galois group. Let $H$ be a non-cyclic abelian subgroup of $G$. Without loss of generality, we may assume $H=\Zz/q\Zz \times \Zz/q\Zz$ for some prime number $q$. Let $\zeta_q$ be a primitive $q$-th root of unity, and let $\fp$ be a non-zero prime ideal of $R$, not in $\mathcal{S}_{{\rm{exc}}}$, that is totally split in $F(\zeta_q)/k$. As $\fp$ is totally split in $k(\zeta_q)/k$, there exists a finite Galois extension $L^{(\fp)}/k_\fp$ with Galois group $H$ \footnote{Indeed, one can take $L^{(\fp)}$ to be the compositum of the fields $L_1^{(\fp)}$ and $L_2^{(\fp)}$, where $L_1^{(\fp)}/k_\fp$ is the unique degree $q$ unramified extension of $k_\fp$, and $L_2^{(\fp)}/k_\fp$ is a finite Galois extension with Galois group $\Zz/q\Zz$ that is totally ramified (such an extension exists; see, e.g., \cite[Chapter IV]{Ser79}).}. In particular, $L^{(\fp)}/k_\fp$ cannot occur as the completion at $\fp$ of a specialization of any of the extensions $E_1/k(T), \dots, E_s/k(T)$, completing the proof.
\end{proof}

\begin{remark}
Our method cannot provide more examples of finite groups such that the answer to Question \ref{question} is negative. Indeed, it requires the following assumption on $G$:

\vspace{1mm}

\noindent
($*$) {\textit{$G$ has a non-cyclic subgroup $H$ such that, for every number field $F$ containing $k$, there exist infinitely many non-zero prime ideals $\fp$ of $R$ such that $\fp$ is totally split in $F/k$, and such that $H$ occurs as a Galois group over $k_\fp$.}}

\vspace{1mm}

\noindent
As shown in the proof of Theorem \ref{thm:neggrunwald}, condition ($*$) holds if $G$ has a non-cyclic abelian subgroup. However, the converse is true, as it is easy to see that any subgroup $H$ of $G$ as in condition ($*$) has to be abelian.
\end{remark}

As an immediate consequence of Theorem \ref{thm:neggrunwald}, the answer to Question \ref{question} is always negative for finite non-cyclic abelian groups, dihedral groups $D_n$ (of order $2n$) with $n$ even, symmetric groups $S_n$ with $n \geq 4$, alternating groups $A_n$ with $n \geq 4$. In Corollary \ref{simple} below, we show that the same is true for arbitrary finite non-abelian simple groups:

\begin{corollary} \label{simple}
Assume that $G$ is a finite non-abelian simple group. Then, the answer to Question \ref{question} is negative for all finite sets of $k$-regular $G$-extensions of $k(T)$.
\end{corollary}

\begin{proof}
By Theorem \ref{thm:neggrunwald}, it suffices to show that $G$ has a non-cyclic abelian subgroup. Suppose by contradiction that $G$ does not. Then, by, e.g., \cite[Chapter XI, Theorem 11.6]{CE56}, every Sylow subgroup of $G$ is either cyclic or a generalized quaternion group. If every 2-Sylow subgroup of $G$ was cyclic, then, every Sylow subgroup of $G$ would then be cyclic. As H\"older proved in 1895 (see, e.g., \cite[Theorem 7.53]{Rot95} for more details), $G$ would be solvable\footnote{More precisely, any given finite group whose every Sylow subgroup is cyclic is metacyclic; see, e.g., \cite[page 290]{Rob96}.}, which cannot happen. One may then apply \cite[Theorem C]{Suz55} to get that $G$ has a normal subgroup $H$ which satisfies the following two conditions:
\begin{itemize}
\item[(1)] $(G:H) \leq 2$,
\item[(2)] $H = H' \times {\rm{SL}}_2(\mathbb{F}_p)$ for some prime number $p$ and some subgroup $H'$ of $H$ whose Sylow subgroups are cyclic. 
\end{itemize}
If $(G:H) = 2$, then, $H$ is trivial as $G$ is simple. Hence, $G$ has order 2, which cannot happen. Then, by (1), we get $G = H' \times {\rm{SL}}_2(\mathbb{F}_p)$, with $H'$ and $p$ as in (2). As $G$ is simple and ${\rm{SL}}_2(\mathbb{F}_p) \not= \{1\}$, we get $H'=\{1\}$, i.e., $G={\rm{SL}}_2(\mathbb{F}_p)$. Hence, ${\rm{SL}}_2(\mathbb{F}_p)$ is simple, which cannot happen as ${\rm{SL}}_2(\mathbb{F}_2) \cong S_3$, and, for $p \geq 3$, the center of ${\rm{SL}}_2(\mathbb{F}_p)$ has order 2.
\end{proof}

Solvability of Grunwald problems via specialization is being investigated further by the first author, in a context of infinite families of regular extensions. See \cite{Koe17b}.

\section{Application to finite parametric sets} \label{sec:parametric}

This section is devoted to our application to the non-existence of finite parametric sets over number fields, as already mentioned in \S\ref{sec:intro_param}. In \S\ref{sec:parametric_criterion}, we state and prove Theorem \ref{grunwald_nonparam}, which is our new general criterion to provide examples of finite groups with no finite parametric set over number fields. Explicit examples, including those given in Theorem \ref{parametric_boohoo} from the introduction (\S\ref{sec:intro_param}), are then given in \S\ref{sec:parametric_examples}.

For this section, let $k$ be a number field, $R$ its ring of integers, $T$ an indeterminate over $k$, and $G$ a finite group.

Let us recall the following definition \cite{KL18}:

\begin{definition} \label{def para}
We say that a set $S$ of $k$-regular $G$-extensions of $k(T)$ is {\textit{parametric}} if every $G$-extension of $k$ occurs as a specialization of some extension $E/k(T)$ in $S$.
\end{definition}

It is shown in \cite{KL18} that certain finite groups do not possess finite parametric sets over the given number field $k$. However, the ``global" strategy developed in that paper requires such finite groups to have a non-trivial proper normal subgroup which satisfies some further properties. In particular, this cannot be used for finite simple groups. 

\subsection{A new general criterion for non-parametricity} \label{sec:parametric_criterion}

Here, we rather use a ``local" approach. Namely, in Theorem \ref{grunwald_nonparam} below, we provide a sufficient condition, based on the proof of Theorem \ref{thm:neggrunwald}, for the finite group $G$ to have no finite parametric set over $k$.

Given a prime number $q$, let $\zeta_q$ denote a primitive $q$-th root of unity.

\begin{theorem} \label{grunwald_nonparam}
Suppose that the following condition holds:

\vspace{1mm}

\noindent
{\rm{($**$)}} there exists a prime number $q$ and a number field $F$ containing $k(\zeta_q)$ such that, for all but finitely non-zero prime ideals $\fp$ of $R$ which are totally split in $F/k$, there exists a $G$-extension of $k$ whose completion at $\fp$ has Galois group $\Zz/q\Zz \times \Zz/q\Zz$. 

\vspace{1mm}

\noindent
Then, given a finite set $S$ of $k$-regular $G$-extensions of $k(T)$, there exist infinitely many $G$-extensions of $k$ each of which is a specialization of no extension $E/k(T)$ in $S$. In particular, $G$ has no finite parametric set over $k$.
\end{theorem}

\begin{proof}
Given a positive integer $s$, let $E_1/k(T), \dots, E_s/k(T)$ be $k$-regular $G$-extensions. Pick a prime number $q$ and a number field $F$ as in condition {\rm{($**$)}}. In particular, $\Zz/q\Zz \times \Zz/q\Zz$ occurs as a subgroup of $G$.
By Proposition \ref{inf_many_cyclic}, there exist infinitely many prime ideals $\fp$ of $R$ which are totally split in $F/k$,
and such that the Galois group of the completion at $\fp$ of any specialization of $E_i/k(T)$ is cyclic ($i=1,\dots,s$). Indeed, one may take all prime ideals $\fp$ which are totally split in the compositum of $F$ and all the residue extensions at branch points of $E_i/k(T)$, $i=1,\dots,s$ (up to finitely many exceptions depending only on $E_1/k(T), \dots, E_s/k(T)$). In particular, for such a $\fp$, a $G$-extension of $k$ as in condition {\rm{($**$)}} is a specialization of $E_i/k(T)$ for no $i \in \{1,\dots,s\}$, completing the proof.
\end{proof}

\subsection{Explicit examples} \label{sec:parametric_examples}

Corollary \ref{examples} below contains our main examples of finite groups with no finite parametric set over number fields:

\begin{corollary} \label{examples}
Assume that either one of the following two conditions holds:
\begin{itemize}
\item[(1)] $G$ has a non-cyclic abelian subgroup, and there exists a finite set $\mathcal{S}_{\rm{exc}}$ of non-zero prime ideals of $R$ such that every Grunwald problem $(G, (L^{(\fp)}/k_\fp)_{\fp\in \mathcal{S}})$, with $\mathcal{S}$ disjoint from $\mathcal{S}_{\rm{exc}}$, has a solution.
\item[(2)] $G=A_n$ ($n \geq 4$).
\end{itemize}
Then, $G$ has no finite parametric set over $k$.
\end{corollary}

\begin{proof}
In either case, it suffices to show that condition ($**$) of Theorem \ref{grunwald_nonparam} holds. This is clear in case (1). Now, assume that we are in case (2). Then, $\Zz/2\Zz \times \Zz/2\Zz$ occurs as a subgroup of $G$. It then suffices to show that, for all prime ideals $\fp$ of $R$, there exists a $G$-extension of $k$ whose completion at $\fp$ has Galois group $\Zz/2\Zz \times \Zz/2\Zz$. Let $\fp$ be a prime ideal of $R$, and let $L^{(\fp)}/k_\fp$ be a $\Zz/2\Zz \times \Zz/2\Zz$-extension. Pick a $\Zz/2\Zz \times \Zz/2\Zz$-extension $L/k$ whose completion at $\fp$ is $L^{(\fp)}/k_\fp$. By a classical result of Mestre (see \cite{Mes90} and \cite[Theorem 3]{KM01}), $L/k$ occurs as a specialization of some $k$-regular $A_n$-extension $E/k(T)$. More precisely, there is a polynomial $P(T,Y) \in k[T][Y]$ which is monic and separable in $Y$, with splitting field $E$ over $k(T)$, and such that $P(0,Y)$ is separable with splitting field $L$ over $k$. Lemma \ref{PV} now ensures the existence of $t_0\in k$ such that the specialization of $E/k(T)$ at $t_0$ has Galois group $A_n$ and has the same completion at $\fp$ as $L/k$. This completes the proof.
\end{proof}

\begin{remark} \
%\begin{itemize}
%\item[(1)] 
(1) Aside from alternating groups, %(and possibly a finite list of isolated groups), 
we are not aware of any other infinite family of non-abelian simple groups which have been shown to satisfy condition ($**$) of Theorem \ref{grunwald_nonparam}. However, by using the same tools, the following variant is derived:

\vspace{1mm}

\noindent
{\textit{Let $G$ be a finite non-abelian simple group, and let $p$ be a prime number such that $H:=\Zz/p\Zz \times \Zz/p\Zz$ occurs as a subgroup of $G$ \footnote{Such a prime number $p$ exists by the proof of Corollary \ref{simple}.}. Then, given a finite set $S$ of $k$-regular $G$-extensions of $k(T)$, there exist infinitely many $H$-extensions of $k$ each of which is a specialization of $E/k(T)$ for no $E/k(T) \in S$.}}
%\item[(2)] 

\vspace{0.5mm}

\noindent (2) Under the expectation of \cite[\S1]{Har07} mentioned in \S\ref{sec:intro_grunwald}, condition (1) of Corollary \ref{examples} holds (and then the conclusion that $G$ has no finite parametric set over $k$ holds as well), provided $G$ is solvable and has a non-cyclic abelian subgroup.
%\end{itemize}
\end{remark}

\bibliography{Biblio2}
\bibliographystyle{alpha}

\end{document}